\documentclass{article}

\usepackage{amssymb}
\usepackage{latexsym}
\usepackage{theorem}

\newenvironment{proof}{\par \noindent{\bf Proof: }}{\hspace{\stretch{1}} $\Box$ \par \mbox{}}
\newcommand{\noproof}{\hspace{\stretch{1}} $\Box$}
\newtheorem{theorem}{Theorem}[section]
\newtheorem{proposition}[theorem]{Proposition}
\newtheorem{lemma}[theorem]{Lemma}
\newtheorem{corollary}[theorem]{Corollary}
{\theorembodyfont{\rmfamily}
\newtheorem{definition}[theorem]{Definition}

}
\newenvironment{theorem*}{\par \medskip \noindent{\bf Theorem }}{\par \mbox{}}
\newenvironment{lemma*}{\par \medskip \noindent{\bf Theorem }}{\par \mbox{}}

\newcommand{\Hom}{\mathop{Hom}}

\newcommand{\Ob}{\mathop{Ob}}

\newcommand{\im}{\mathop{im}}

\newcommand{\F}{\mathbb{F}}

\newcommand{\KK}{{\mathbb K}{\mathbb K}}
\newcommand{\N}{\mathbb {N}}

\title{$KK$-theory spectra for $C^\ast$-categories and discrete groupoid $C^\ast$-algebras}

\author{Paul D. Mitchener}

\begin{document}

\maketitle

\tableofcontents

\subsection*{Abstract}

In this paper we refine a version of bivariant $K$-theory developed
by Cuntz to define symmetric spectra representing the $KK$-theory of
$C^\ast$-categories and discrete groupoid $C^\ast$-algebras.  In
both cases, the Kasparov product can be expressed as a smash product
of spectra.

\section{Introduction}

In \cite{Cu3}, J.Cuntz developed $KK$-theory for locally convex
algebras in order to look at versions of the Chern character for
bivariant theories.  This approach, using the thesis of A.B.Thom,
was simplified in \cite{Cu4}.  In an unpublished preprint, M.Joachim
and S.Stolz have used Cuntz's approach to $KK$-theory to define
symmetric $KK$-theory spectra for $C^\ast$-algebras.

The purpose of this article is to go in a slightly different
direction with the $KK$-theory machinery by looking at the
$KK$-theory of $C^\ast$-categories and of discrete groupoid
$C^\ast$-algebras.  In both of these cases, the theory can naturally
be expressed in terms of symmetric spectra, and the Kasparov product
can be realised at the level of spectra.

Thus, for $C^\ast$-categories $\mathcal A$ and $\mathcal B$ (or as a
special case $C^\ast$-algebras), we have a symmetric spectrum $\KK
({\mathcal A},{\mathcal B})$ representing $KK$-theory.  If we are
working over the complex numbers, this spectrum is a symmetric $\KK
({\mathbb C} , {\mathbb C})$-module spectrum.  Over the real
numbers, we have a $\KK ({\mathbb R} , {\mathbb R} )$-module
spectrum.  In the special case that ${\mathcal A} = {\mathcal B}$,
the spectrum $\KK ({\mathcal A},{\mathcal A})$ is a symmetric ring
spectrum.

Similar results hold in the equivariant case.  To be precise, if
$\mathcal G$ is a discrete groupoid (or as a special case, a
discrete group), and $A$ and $B$ are $\mathcal G$-$C^\ast$-algebras,
then we have a symmetric spectrum $\KK_{\mathcal G}(A,B)$
representing equivariant $KK$-theory.  This spectrum is a symmetric
$\KK_{\mathcal G}({\mathbb C} , {\mathbb C})$-module spectrum in the
complex case, and a symmetric $\KK_{\mathcal G}({\mathbb R}
,{\mathbb R} )$-module spectrum in the real case.  The spectrum
$\KK_{\mathcal G}(A,A)$ is a symmetric ring spectrum.

There are a several potential applications of the new machinery. The
constructions in this article are both simpler and have more
structure than the $KK$-theory spectra constructed in
\cite{Mitch3,Mitch6}, where $KK$-theory spectra for
$C^\ast$-categories are developed in order to examine analytic
assembly maps.  The extra structure present should be useful when
homotopy-theoretic arguments involving $KK$-theory are applied, for
example (see \cite{Stolz1}) in the proof that the Baum-Connes
conjecture implies the stable Gromov-Lawson-Rosenberg conjecture.

\section{$C^\ast$-categories} \label{prelim}

Let $\F$ denote either the field of real numbers or the field of
complex numbers.  Recall (see for example \cite{Mi}) that a small
category $\mathcal A$ is called an {\em unital algebroid} (over the
field $\F$) if each morphism set $\Hom (a,b)_{\mathcal A}$ is a
vector space over the field $\F$ and composition of morphisms is
bilinear.

An {\em involution} on a unital algebroid $\mathcal A$ is a collection of maps
\[
\Hom (a,b)_{\mathcal A}\rightarrow \Hom (b,a)_{\mathcal A}
\]
written $x\mapsto x^\ast$ such that:

\begin{itemize}

\item $(\alpha x + \beta y)^\ast = \overline{\alpha}x^\ast +
\overline{\beta}y^\ast$ for all scalars $\alpha ,\beta \in \F$ and
morphisms $x,y \in \Hom (a,b)_{\mathcal A}$.

\item $(xy)^\ast = y^\ast x^\ast$ for all composable morphisms $x$
and $y$.

\item $(x^\ast )^\ast =x$ for every morphism $x$.

\end{itemize}

Given unital algebroids with involution, $\mathcal A$ and $\mathcal
B$, we call a functor $F\colon {\mathcal A}\rightarrow {\mathcal B}$
a {\em $\ast$-functor} if each map $F\colon \Hom (a,b)_{\mathcal
A}\rightarrow \Hom (F(a),F(b))_{\mathcal B}$ is linear, and
$F(x^\ast ) = F(x)^\ast$ for each morphism $x$ in the category
$\mathcal A$.

A {\em non-unital} algebroid with involution is a collection of
objects, morphisms, and maps similar to the above, except that there
need not exist identity morphisms $1\in \Hom (a,a)_{\mathcal A}$.
Thus, a non-unital algebroid with involution is no longer a
category, but rather a slightly more general object which could be
termed a {\em non-unital category}.    We can similarly define
$\ast$-functors between non-unital algebroids with involution.

In general, when we talk about $\ast$-algebroids and $\ast$-functors, we need to allow the possibility that they may be non-unital.

\begin{definition}
Let $\mathcal A$ be an algebroid with involution.  Then we call
$\mathcal A$ a {\em pre-$C^\ast$-category} if each morphism set is a
normed vector space and the following three axioms hold:

\begin{itemize}

\item Let $x$ and $y$ be composable morphisms in $\mathcal A$.  Then $\| xy \| \leq \| x\| \| y\|$.

\item Let $x\in \Hom (a,b)_{\mathcal A}$.  Then the product $x^\ast
x$ is a positive element of the algebra $\Hom (a,a)_{\mathcal
A}$.\footnote{That is to say the spectrum is a subset of the
positive real numbers.}

\item The {\em $C^\ast$-identity} $\| x^\ast x \| = \| x\|^2$ holds
for any morphism, $x$, in the category $\mathcal A$.

\end{itemize}

A collection of norms on the morphism sets of an algebroid with
involution that turns it into a pre-$C^\ast$-category is called a
{\em $C^\ast$-norm}.  There is a corresponding definition of a {\em
$C^\ast$-seminorm}.  A pre-$C^\ast$-category is called a {\em
$C^\ast$-category} if every morphism set is complete.
\end{definition}

In the above definition, a pre-$C^\ast$-category or
$C^\ast$-category could be non-unital (and thus no longer, strictly
speaking, a category).  Such $C^\ast$-categories are referred to in
\cite{Kan3} as {\em $C^\ast$-categeroids}; we will not use this
terminology here.

A $C^\ast$-algebra can be regarded as a $C^\ast$-category with only
one object. Conversely, $C^\ast$-categories and $\ast$-functors have
a number of useful properties similar to those of $C^\ast$-algebras
and $\ast$-homomorphisms.  For example, we have the following.

\begin{proposition} \label{AFprop}
Let $\mathcal A$ and $\mathcal B$ be $C^\ast$-categories.  Let
$F\colon {\mathcal A}\rightarrow {\mathcal B}$ be a $\ast$-functor.
Then:

\begin{itemize}

\item Each map $F\colon \Hom (a,b)_{\mathcal A}\rightarrow \Hom
(F(a),F(b))_{\mathcal B}$ is norm-decreasing.

\item Each map $F\colon \Hom (a,b)_{\mathcal A}\rightarrow \Hom
(F(a),F(b))_{\mathcal B}$ has closed image.

\item Suppose that the $\ast$-functor $F$ is faithful.  Then each
map  $F\colon \Hom (a,b)_{\mathcal A}\rightarrow \Hom
(F(a),F(b))_{\mathcal B}$ is an isometry.

\end{itemize}

\noproof
\end{proposition}

Perhaps the most important elementary result on $C^\ast$-categories
is the following; see \cite{GLR} or \cite{Mitch2} for further
details.

\begin{theorem} \label{GNS}
Let $\mathcal A$ be a $C^\ast$-category.  Let $\mathcal L$ be the
$C^\ast$-category of all Hilbert spaces and bounded linear
maps.\footnote{Strictly speaking, the category $\mathcal L$ is not a
$C^\ast$-category since it is not small.  This problem does not
matter to us since we will not be doing constructions directly
involving the category $\mathcal L$; we can always pick a small full
subcategory.}  Then there exists a faithful $\ast$-functor $\rho
\colon {\mathcal A}\rightarrow {\mathcal L}$. \noproof
\end{theorem}

A $\ast$-functor $\rho \colon {\mathcal A}\rightarrow {\mathcal L}$
is called a {\em representation} of $\mathcal A$.  We write
${\mathcal L}(H_1,H_2)$ to denote the Banach space of all bounded
linear maps from a Hilbert space $H_1$ to a Hilbert space $H_2$.

Although most of the time, we will be looking at $\ast$-functors between $C^\ast$-categories, occasionally we need a slightly weaker notion.

\begin{definition}
Let $\mathcal A$ and $\mathcal B$ be $C^\ast$-categories.  Then a
{\em completely positive map} $q\colon {\mathcal A}\rightarrow
{\mathcal B}$ consists of a map $q\colon \Ob ({\mathcal
A})\rightarrow \Ob ({\mathcal B})$ along with a collection of linear
maps $q\colon \Hom (a,b)_{\mathcal A}\rightarrow \Hom
(q(a),q(b))_{\mathcal B}$ such that the sum
$$\sum_{i,j=1}^n y_i^\ast q (x_i^\ast x_j )y_j$$
is a positive element of the $C^\ast$-algebra $\Hom
(q(a),q(a))_{\mathcal B}$ for all morphisms $x_i \in \Hom
(b,c)_{\mathcal A}$ and $y_j\in \Hom (q(a),q(b))_{\mathcal B}$.
\end{definition}

Any $\ast$-functor is clearly a completely positive map.  However,
in general, a completely positive map is not even a functor.

We have a version of Stinespring's theorem (see for example chapter
3 of \cite{HR1}) in the world of $C^\ast$-categories.

\begin{theorem}
Let $\mathcal A$ be a unital $C^\ast$-category, and for each object
$a\in \Ob ({\mathcal A})$, let $H_a$ be an associated Hilbert space.

Then a set of unit-preserving linear maps $q\colon \Hom
(a,b)_{\mathcal A}\rightarrow {\mathcal L}(H_a,H_b)$ is completely
positive if and only if we have:

\begin{itemize}

\item A unital representation $\rho \colon {\mathcal A}\rightarrow
{\mathcal L}$.

\item A Hilbert space isometry $V_a \colon H_a\rightarrow \rho (a)$
for each object $a\in \Ob ({\mathcal A})$.

\end{itemize}

such that $q(x) = V_b^\ast \rho (x) V_a$ for all $x\in \Hom
(a,b)_{\mathcal A}$.
\end{theorem}

\begin{proof}
Suppose we have a representation $\rho$ and isometries $V_a$ as
described above.  Let $x_i \in \Hom (b,c)_{\mathcal A}$ and $y_j \in
{\mathcal L}(H_a,H_b)$.  Then
$$\sum_{i,j=1}^n y_i^\ast q(x_i^\ast x_j)y_j = \left( \sum_{i=1}^n
x_i V_b y_i \right)^\ast \left( \sum_{i=1}^n x_i V_b y_i \right)$$
so the set of maps $q$ forms a completely positive map.

Conversely, let $q$ be completely positive.  On the algebraic tensor
product of vector spaces $\oplus_{a\in \Ob ({\mathcal A})} \Hom
(a,b)_{\mathcal A} \odot H_b$, define a sesquilinear form by the
formula
$$\langle \sum_{i=1}^m x_i \otimes v_i , \sum_{j=1}^n y_j \otimes
w_j \rangle = \sum_{i,j=1}^{m,n} \langle v_i , q(x_i^\ast y_j ) w_j
\rangle$$

Since $q$ is completely positive, the above sesquilinear form is
positive definite.  We can therefore take the quotient by the set of
tensors $\eta$ such that $\langle \eta ,\eta \rangle =0$ to form an
inner product space, and complete to form a Hilbert space
$\tilde{H}_a$.  Let us write $[x\otimes v]$ to denote the
equivalence class in this space of a tensor $x\otimes v$.

Define a representation $\rho \colon {\mathcal A}\rightarrow
{\mathcal L}$ by writing $\rho (a) = \tilde{H}_a$ for each object
$a\in \Ob ({\mathcal A})$, and
$$\rho (x)[y\otimes v] = [xy \otimes v]$$
where $x\in \Hom (b,c)_{\mathcal A}$, $y\in \Hom (a,b)_{\mathcal
A}$, and $v\in H_b$.

Define a Hilbert space isometry $V_b \colon H_b \rightarrow
\tilde{H}_b$ by the formula $V_b (v) = [1_b \otimes v]$.  Observe
that, for $x\in \Hom (a,b)_{\mathcal A}$ and $v\in H_a$, we have
$$\rho (x) V_a (v) =[x\otimes v]$$
and, for $w\in H_b$:
$$\langle w, q(x)v \rangle = \langle [x\otimes v],[1\otimes w]
\rangle$$

Hence $V_b^\ast [x\otimes v] = q(x)v$.  We see that
$$V_b^\ast \rho (x) V_a = q(x)$$
and we are done.
\end{proof}

\section{Short Exact Sequences and Tensor Products}

\begin{definition}
A sequence
\[
{\mathcal I}\stackrel{i}{\rightarrow} {\mathcal E} \stackrel{j}{\rightarrow} {\mathcal B}
\]
of $C^\ast$-categories and $\ast$-functors is termed a {\em short exact sequence} if:

\begin{itemize}

\item The categories $\mathcal I$, $\mathcal E$, and $\mathcal B$
have the same set of objects, and the functors $i$ and $j$ act as
the identity map on the set of objects.

\item Each sequence of vector spaces
\[
0\rightarrow \Hom (a,b)_{\mathcal I}\stackrel{i}{\rightarrow} \Hom
(a,b)_{\mathcal E}\stackrel{j}{\rightarrow} \Hom (a,b)_{\mathcal
B}\rightarrow 0
\]
is a short exact sequence.

\end{itemize}

\end{definition}

We will generally write
\[
0\rightarrow {\mathcal I}\stackrel{i}{\rightarrow} {\mathcal E} \stackrel{j}{\rightarrow} {\mathcal B} \rightarrow 0
\]
when we have a short exact seqence.  A short exact sequence is
termed {\em split exact} if it comes equipped with a $\ast$-functor
$s\colon {\mathcal B}\rightarrow {\mathcal E}$ such that $j\circ s
=1_{\mathcal B}$.  We term the short exact sequence {\em semi-split}
if there is a norm-decreasing completely positive map $q\colon
{\mathcal A}\rightarrow {\mathcal B}$ such that $q\circ s = 1_{\Hom
(a,b)_{\mathcal B}}$.

The above $\ast$-functor or completely positive map $s\colon
{\mathcal B}\rightarrow {\mathcal E}$ is called a {\em splitting} in
the first case, and a {\em completely positive splitting} in the
second case.

\begin{definition}
Let $\mathcal A$ be a $C^\ast$-category.  Then we define the
category ${\mathcal A}[0,1]$ to be the $C^\ast$-category with the
same set of objects as $\mathcal A$, where the morphism set $\Hom
(a,b)_{{\mathcal A}[0,1]}$ consists of all continuous functions
$f\colon [0,1] \rightarrow \Hom (a,b)_{\mathcal A}$. The norm on the
space $\Hom (a,b)_{{\mathcal A}[0,1]}$ is the supremum norm.

We define the {\em cone}, $C{\mathcal A}$, to be the subcategory
 of ${\mathcal A}[0,1]$ with the same set of objects, and morphism sets
\[
\Hom (a,b)_{C{\mathcal A}} = \{ f\in \Hom (a,b)_{{\mathcal A}[0,1]} \ |\ f(0)=0 \}
\]

The {\em suspension}, $\Sigma {\mathcal A}$ is the subcategory of
$C{\mathcal A}$ with the same set of objects, and morphism sets
\[
\Hom (a,b)_{\Sigma {\mathcal A}} = \{ f\in \Hom (a,b)_{C{\mathcal A}} \ |\ f(1)=0 \}
\]
\end{definition}

We have a canonical inclusion $\ast$-functor $i\colon \Sigma
{\mathcal A} \rightarrow C{\mathcal A}$.  There is also a
$\ast$-functor $j\colon C{\mathcal A}\rightarrow {\mathcal A}$,
defined to be the identity on the set of objects, and by the formula
$j(f) = f(1)$ for each morphism $f\in \Hom (a,b)_{C{\mathcal A}}$.
It is easy to check that we have a short exact sequence
\[
0 \rightarrow \Sigma {\mathcal A} \rightarrow C{\mathcal A}\rightarrow {\mathcal A}\rightarrow 0
\]

The above short exact sequence is semi-split; we have a completely
positive splitting $s\colon {\mathcal A}\rightarrow C{\mathcal A}$,
defined to be the identity the set of objects, and by the formula
\[
s (x)(t)= tx \qquad t\in [0,1],\ x\in \Hom (a,b)_{\mathcal A}
\]

Further, the above semi-split short exact sequence is natural in the
sense that the assignments ${\mathcal A}\mapsto C{\mathcal A}$ and
${\mathcal A}\mapsto \Sigma {\mathcal A}$ are $\ast$-functors
depending functorially on the $C^\ast$-category $\mathcal A$, and
the maps $i$, $j$, and $s$ are natural transformations. Given a
$\ast$-functor $\alpha \colon {\mathcal A}\rightarrow {\mathcal B}$,
we write $\Sigma \alpha \colon \Sigma {\mathcal A}\rightarrow \Sigma
{\mathcal B}$ to denote the induced $\ast$-functor.

\begin{definition}
Let $\mathcal A$ and $\mathcal B$ be $C^\ast$-categories.  Then we
define the {\em algebraic tensor product} ${\mathcal A}\odot
{\mathcal B}$ to be the category with the set of objects
\[
\Ob ({\mathcal A}\odot {\mathcal B}) = \{ a\otimes b \ |\ a\in \Ob
({\mathcal A}), b\in \Ob ({\mathcal B}) \}
\]
where the morphism set $\Hom (a\otimes b,a'\otimes b')_{{\mathcal
A}\odot {\mathcal B}}$ is a the algebraic tensor product of vector
spaces $\Hom (a,b)_{\mathcal A}\odot \Hom (a',b')_{\mathcal B}$.
\end{definition}

By theorem \ref{GNS}, we can find faithful $\ast$-functors
$\rho_{\mathcal A} \colon {\mathcal A}\rightarrow {\mathcal L}$ and
$\rho_{\mathcal B}\colon {\mathcal A}\rightarrow {\mathcal L}$.  By
considering the tensor product of Hilbert spaces, we obtain a
$\ast$-functor $\rho_{\mathcal A}\otimes \rho_{\mathcal B} \colon
{\mathcal A}\odot {\mathcal B}\rightarrow {\mathcal L}$.  We define
the {\em spatial tensor product}, ${\mathcal A}\otimes {\mathcal
B}$, of the $C^\ast$-categories $\mathcal A$ and $\mathcal B$ to be
the completion of the algebraic tensor product with respect to the
norm $\| u \| := \| \rho_{\mathcal A}\otimes \rho_{\mathcal B} (u)
\|$.

The spatial tensor product is well-defined, and does not depend on
the choice of faithful representation; for more details, see
\cite{Mitch2}.

For any $C^\ast$-category $\mathcal A$, the tensor product
$C[0,1]\otimes {\mathcal A}$ is naturally isomorphic to the category
${\mathcal A}[0,1]$.  The proof is the same as that of the
corresponding result for $C^\ast$-algebras\footnote{The proof works
for any sensible tensor product of $C^\ast$-algebras since the
$C^\ast$-algebra $C[0,1]$ is commutative and therefore nuclear.};
see for example appendix T of \cite{W-O}.  The following result for
cones and suspensions follows.

\begin{proposition} \label{cstp}
Let $\mathcal A$ and $\mathcal B$ be $C^\ast$-categories.  Then we
have natural isomorphisms
\[
C({\mathcal A}\otimes {\mathcal B})\cong {\mathcal A}\otimes
C{\mathcal B} \cong (C{\mathcal A})\otimes {\mathcal B}
\]
and
\[
\Sigma ({\mathcal A}\otimes {\mathcal B})\cong {\mathcal A}\otimes
\Sigma{\mathcal B} \cong (\Sigma {\mathcal A})\otimes {\mathcal B}
\]
\noproof
\end{proposition}

\begin{lemma}
Let $q\colon {\mathcal A}\rightarrow {\mathcal B}$ be a completely
positive map, and let $\mathcal C$ be a $C^\ast$-category.  Then the
obvious map
$$q\otimes 1 \colon {\mathcal A}\otimes {\mathcal C} \rightarrow
{\mathcal B}\otimes {\mathcal C}$$ is completely positive.
\end{lemma}

\begin{proof}
By choosing representations and adjoining units if necessary, we can
assume that $\mathcal B$ and $\mathcal C$ are unital subcategories
of $\mathcal L$, and the completely positive map $q$ is
unit-preserving.

Let $a\in \Ob ({\mathcal A})$, and write $q(a)=H_a$.  Then by
Stinespring's theorem, we have a unital $\ast$-functor $\rho \colon
{\mathcal A}\rightarrow {\mathcal L}$, and Hilbert space isometries
$V_a \colon H_a \rightarrow \rho (a)$ such that
$$q(x) = V_{a'}^\ast \rho (x) V_a$$
for all $x\in \Hom (a,a')_{\mathcal A}$.

Consider a Hilbert space $H_c \in \Ob ({\mathcal C})$.  Then we have
a unital $\ast$-functor $\tilde{\rho }\colon {\mathcal C}
\rightarrow {\mathcal L}$ defined by writing
$$\tilde{\rho} (a\otimes H_c ) \qquad a\in \Ob ({\mathcal A}),\ H_c
\in \Ob ({\mathcal C})$$ and
$$\tilde{\rho} (x\otimes y) = \rho (x) \otimes y \qquad x\in \Hom
(a,a')_{\mathcal A},\ y\in \Hom (H_c,H_{c'})_{\mathcal C}$$

Define isometries $\tilde{V}_{a\otimes H_c} \colon H_a\otimes H_c
\rightarrow \rho (a) \otimes H_c$ by the formula
$$\tilde{V}_{a\otimes H_c} (v\otimes w) = V_a (v) \otimes w$$

Then we have, for $x\in \Hom (a,a')_{\mathcal A}$ and $y\in \Hom
(H_c,H_{c'})_{\mathcal C}$
$$\tilde{V}_{a' \otimes H_{c'}}^\ast \tilde{\rho}(x\otimes
y)V_{a\otimes H_c} = q(x) \otimes y$$

Therefore, by Stinespring's theorem, the collection $q\otimes 1$ is
a completely positive map, as required.
\end{proof}

Note that the above proof relies on the fact that we are using the
spatial tensor product.

\begin{theorem} \label{sstens}
Let
\[
0\rightarrow {\mathcal I}\stackrel{i}{\rightarrow} {\mathcal E}
\stackrel{j}{\rightarrow} {\mathcal B}\rightarrow 0
\]
be a semi-split exact sequence.  Let $\mathcal C$ be any
$C^\ast$-category.  Then we have a semi-split exact sequence
\[
0\rightarrow {\mathcal I}\otimes {\mathcal C}\stackrel{i\otimes
1}{\rightarrow} {\mathcal E}\otimes {\mathcal C} \stackrel{j\otimes
1}{\rightarrow} {\mathcal B}\otimes {\mathcal C}\rightarrow 0
\]
\end{theorem}

\begin{proof}
Let $q\colon {\mathcal B}\rightarrow {\mathcal E}$ be a completely
positive map such that $j\circ q =1_{\mathcal B}$.  Then by the
above lemma, we have a completely positive map $q\otimes 1 \colon
{\mathcal B}\otimes {\mathcal C} \rightarrow {\mathcal E}\otimes
{\mathcal C}$ such that $(j\otimes 1)(q\otimes 1)= 1_{{\mathcal
B}\otimes {\mathcal C}}$.

The map $i\otimes 1$ is certainly injective, and by the above, the
map $j\otimes 1$ is surjective.  We know that $\im i = \ker j$.
Hence $ij=0$, and $(i\otimes 1)(j\otimes 1)=0$.  It follows that
$\im (i\otimes 1) \subseteq \ker (j\otimes 1)$.

Again using the fact that $\im i = \ker j$, we see that each
morphism set of the image $\im (i\otimes 1)$ is a dense subset of
the corresponding morphism set of the $C^\ast$-category $\ker
(j\otimes 1)$.  But $i\otimes 1$ is a $\ast$-functor, so by
proposition \ref{AFprop}, each morphism set of the image $\im
(i\otimes 1)$ is closed.  It follows that $\im (i\otimes 1) = \ker
(j\otimes 1)$, and the sequence
\[
0\rightarrow {\mathcal I}\otimes {\mathcal C}\stackrel{i\otimes
1}{\rightarrow} {\mathcal E}\otimes {\mathcal C} \stackrel{j\otimes
1}{\rightarrow} {\mathcal B}\otimes {\mathcal C}\rightarrow 0
\]
is exact.  The equation $(j\otimes 1)(q\otimes 1)= 1_{{\mathcal
B}\otimes {\mathcal C}}$ now tells us that the above short exact
sequence is semi-split, as required.
\end{proof}

\begin{definition}
Let $\mathcal A$ be a $C^\ast$-category.  Given objects $a,b\in \Ob ({\mathcal A})$, let us define
\[
\Hom (a,b)^{\odot (k+1)}_{\mathcal A} = \bigoplus_{c_i \in \Ob ({\mathcal A})} \Hom (a, c_1)\odot \Hom (c_1,c_2) \odot \cdots \otimes \Hom (c_k,b)
\]

The {\em tensor algebroid}, $T_\mathrm{alg} {\mathcal A}$, is the
algebroid with the same set of objects as the $C^\ast$-category
$\mathcal A$ and morphism sets \[ \Hom
(a,b)_{T_\mathrm{alg}{\mathcal A}} = \bigoplus_{k=1}^\infty \Hom
(a,b)^{\odot k}_{\mathcal A}
\]

Here the the space $\Hom (a,b)^{\otimes 1}_{\mathcal A}$ is simply
the morphism set $\Hom (a,b)_{\mathcal A}$.  Composition of
morphisms in the tensor algebroid is defined by concatenation of
tensors.
\end{definition}

We have a canonical set of linear maps $\sigma \colon {\mathcal
A}\rightarrow T_\mathrm{alg}{\mathcal A}$ defined by mapping each
morphism set of the category $\mathcal A$ onto the first summand.
This collection of linear maps is not compatible with the
composition defined in the two categories, and so does not define a
functor.  The following result is easy to check.

\begin{proposition}
The tensor algebroid, $T_\mathrm{alg}{\mathcal A}$, can be equipped
with a $C^\ast$-norm given by defining the norm, $\| u \|$, of a
morphism $u$ in the tensor algebroid to be the supremum of all
values $p(\varphi (u))$ where $\varphi \colon T_\mathrm{alg}
{\mathcal A}\rightarrow {\mathcal B}$ is a $\ast$-functor into a
$C^\ast$-category $\mathcal B$ such that the composition $\varphi
\circ \sigma \colon {\mathcal A}\rightarrow {\mathcal B}$ is
completely positive and norm-decreasing, and $p$ is a
$C^\ast$-seminorm on the $C^\ast$-category $\mathcal B$. \noproof
\end{proposition}

We define the {\em tensor $C^\ast$-category}, $T{\mathcal A}$, to be
the completion of the tensor algebroid $T_\mathrm{alg}{\mathcal A}$
with respect to the above norm.  Formation of the tensor
$C^\ast$-category defines a functor from the category of
$C^\ast$-categories and $\ast$-functors to itself.

We have a natural $\ast$-functor $\pi \colon T{\mathcal
A}\rightarrow {\mathcal A}$ defined to be the identity on the set of
objects, and by the formula
\[
\varphi (x_1 \otimes \cdots \otimes x_n) = x_n\ldots x_1
\]
for morphisms $x_i \in \Hom (c_i,c_{i+1})$.  It follows that there
is a $C^\ast$-category $J{\mathcal A}$ with the same objects as the
category $\mathcal A$, and morphism sets
\[
\Hom (a,b)_{J{\mathcal A}} = \ker \pi \colon \Hom (a,b)_{T{\mathcal A}} \rightarrow \Hom (a,b)_{\mathcal A}
\]

This category fits into a natural short exact sequence
\[
0\rightarrow J{\mathcal A} \hookrightarrow T{\mathcal A} \stackrel{\pi}{\rightarrow} {\mathcal A}\rightarrow 0
\]

Given a $\ast$-functor $\alpha \colon {\mathcal A}\rightarrow
{\mathcal B}$, we write $J\alpha \colon J{\mathcal A}\rightarrow
J{\mathcal B}$ to denote the induced $\ast$-functor.

\begin{proposition}
The above short exact sequence has a completely positive splitting
$\sigma \colon {\mathcal A}\rightarrow T{\mathcal A}$ defined by
mapping each morphism set of the category $\mathcal A$ onto the
first summand.
\end{proposition}

\begin{proof}
It is obvious that $\sigma \circ \pi = 1_{\mathcal A}$, and that the
map $\sigma$ is compatible with the involution.  The result now
follows by the $C^\ast$-category axiom that says that the product
$x^\ast x$ is positive for any morphism $x$ in a $C^\ast$-category.
\end{proof}

The following now follows from theorem \ref{sstens}.

\begin{corollary} \label{jtens}
Let $\mathcal A$ and $\mathcal C$ be $C^\ast$-categories.  Then we
have a semi-split exact sequence
$$0 \rightarrow (J{\mathcal A})\otimes {\mathcal C} \rightarrow (T{\mathcal A})\otimes {\mathcal C} \stackrel{\pi \otimes 1}{\rightarrow} {\mathcal A}\otimes {\mathcal C} \rightarrow 0$$
\noproof
\end{corollary}

We refer to the following result as the {\em universal property} of
the tensor $C^\ast$-category.

\begin{proposition} \label{Tuniversal}
Let $\alpha \colon {\mathcal A}\rightarrow {\mathcal B}$ be a
norm-decreasing collection of completely positive linear maps
between $C^\ast$-categories.  Then there is a unique $\ast$-functor
$\varphi \colon T{\mathcal A}\rightarrow {\mathcal B}$ such that
$\alpha = \varphi \circ \sigma$.
\end{proposition}

\begin{proof}
We can define a $\ast$-functor $\varphi \colon
T_\mathrm{alg}{\mathcal A}\rightarrow {\mathcal B}$ by writing
$\varphi (a) = \alpha (a)$ for each object $a\in \Ob ({\mathcal
A})$, and
\[
\varphi (x_1 \otimes \cdots \otimes x_n) = \alpha (x_n)\ldots \alpha (x_1)
\]
for morphisms $x_i \in \Hom (C_i,C_{i+1})$.  It is easy to see that
$\varphi$ is the unique $\ast$-functor with the property that
$\alpha =\varphi \circ \sigma$.

By definition of the norm on the tensor algebroid,
$T_\mathrm{alg}{\mathcal A}$, the $\ast$-functor $\varphi$ is
norm-decreasing on each morphism set in the tensor algebroid, and
therefore extends to a $\ast$-functor $\varphi \colon T{\mathcal
A}\rightarrow {\mathcal B}$ by continuity.
\end{proof}

\begin{proposition} \label{classmap}
Let
\[
0\rightarrow {\mathcal I}\rightarrow {\mathcal E} \rightarrow {\mathcal B} \rightarrow 0
\]
be a semi-split short exact sequence of $C^\ast$-categories.  Let
$\alpha \colon {\mathcal A}\rightarrow {\mathcal B}$ be a
$\ast$-functor.  Then there are $\ast$-functors $\tau \colon
T{\mathcal A}\rightarrow {\mathcal E}$ and $\gamma \colon J{\mathcal
A}\rightarrow {\mathcal I}$ such that we have a commutative diagram
\[
\begin{array}{ccccccccc}
0 & \rightarrow & J{\mathcal A} & \rightarrow & T{\mathcal A} & \rightarrow & {\mathcal A} & \rightarrow & 0 \\
& & \downarrow & & \downarrow & & \downarrow \\
0 & \rightarrow & {\mathcal I} & \rightarrow & {\mathcal E} & \rightarrow & {\mathcal B} & \rightarrow & 0 \\
\end{array}
\]
\end{proposition}

\begin{proof}
Let $s\colon {\mathcal B}\rightarrow {\mathcal E}$ be a completely positive splitting.

Then by the universal property of the tensor $C^\ast$-category we
have a unique $\ast$-functor $\tau \colon T{\mathcal A}\rightarrow
{\mathcal E}$ such that $s\circ \alpha = \tau \circ \sigma$.  It
follows that the homomorphism $\tau$ fits into the above diagram.
The homomorphism $\gamma$ is defined by restriction of $\tau$.
\end{proof}

\begin{definition}
The homomorphism $\gamma$ is called the {\em classifying map} of the diagram
\[
\begin{array}{ccccccccc}
& & & & & & {\mathcal A} \\
& & & & & & \downarrow \\
0 & \rightarrow & {\mathcal I} & \rightarrow & {\mathcal E} & \rightarrow & {\mathcal B} & \rightarrow & 0 \\
\end{array}
\]
\end{definition}

\section{$C^\ast$-algebras associated to $C^\ast$-categories}

In \cite{Jo2}, Joachim defined the $K$-theory spectrum of a
$C^\ast$-category by associating a certain $C^\ast$-algebra to a
$C^\ast$-category and then defining the $K$-theory of the
$C^\ast$-category to be the $K$-theory of the associated
$C^\ast$-algebra.  In this section, we make a similar construction
in order to look at $KK$-theory spectra.  However, our
$C^\ast$-algebra will be based on constructions of $K$-theory
spectra in \cite{Mitch2.5} rather than on Joachim's
$C^\ast$-algebra.

\begin{definition}
Let $\mathcal A$ be a $C^\ast$-category.  Then we define the {\em
additive completion}, ${\mathcal A}_\oplus$, to be the algebroid in
which the objects are formal sums $a_1 \oplus \cdots \oplus a_m$,
where $a_i \in \Ob ({\mathcal A})$.  For natural numbers $m,n \in
\N$, the morphism set $\Hom (a_1 \oplus \cdots \oplus a_m , b_1
\oplus \cdots \oplus b_n )$ is the set of matrices
\[
\left\{ \left( \begin{array}{ccc}
x_{1,1} & \cdots & x_{1,m} \\
\vdots & \ddots & \vdots \\
x_{n,1} & \cdots & x_{n,m} \\
\end{array} \right)
\ |\ x_{i,j} \in \Hom (a_j,b_i)_{\mathcal A} \right\}
\]

Composition of matrices is defined by matrix multiplication.  The involution is defined by the formula
\[
\left( \begin{array}{ccc}
x_{1,1} & \cdots & x_{1,m} \\
\vdots & \ddots & \vdots \\
x_{n,1} & \cdots & x_{n,m} \\
\end{array} \right)^\ast
=
\left( \begin{array}{ccc}
x_{1,1}^\ast & \cdots & x_{n,1}^\ast \\
\vdots & \ddots & \vdots \\
x_{1,m}^\ast & \cdots & x_{n,m}^\ast \\
\end{array} \right)
\]
\end{definition}

Given objects $a = a_1 \oplus \cdots \oplus a_m$ and $b= b_1\oplus \cdots \oplus b_n$, we define
\[
a \oplus b = a_1 \oplus \cdots \oplus a_m \oplus b_1\oplus \cdots \oplus b_n
\]

As a special case, we define $0$ to be the formal sum of no objects.
We have morphism sets $\Hom (0,a)_{{\mathcal A}_\oplus } = \{ 0 \}$
and $\Hom (a,0)_{{\mathcal A}_\oplus } = \{ 0 \}$ for each object
$a\in \Ob ({\mathcal A}_\oplus )$.

It is clear that the additive completion of a $C^\ast$-category is
an additive category.  Given a $\ast$-functor $\alpha \colon
{\mathcal A}\rightarrow {\mathcal B}$ where the $C^\ast$-category
$\mathcal B$ is additive, there is an obvious induced additive
functor $\alpha \colon {\mathcal A}_\oplus \rightarrow {\mathcal
B}$.  In particular, given a faithful representation $\rho \colon
{\mathcal A}\rightarrow {\mathcal L}$, there is an induced faithful
representation $\rho_\oplus \colon {\mathcal A}_\oplus \rightarrow
{\mathcal L}$.

We can define a $C^\ast$-norm on the category ${\mathcal A}_\oplus$
by deeming the induced representation $\rho_\oplus$ to be an
isometry.  The category $\mathcal A$ is a $C^\ast$-category with
respect to this norm.  Further, the norm does not depend on the
representation $\rho$.  Thus the additive completion is a functor
from the category of $C^\ast$-categories and $\ast$-functors to the
category of additive $C^\ast$-categories and additive
$\ast$-functors.

Further details of this construction and proofs of the above
statements can by found in \cite{Mitch2.5}.    The following result
is easy to check.

\begin{proposition}
Let $\mathcal A$ be a $C^\ast$-category.  Then we have natural isomorphisms
\[
(\Sigma {\mathcal A})_\oplus \cong \Sigma ({\mathcal A}_\oplus )
\qquad (J{\mathcal A})_\oplus \cong J({\mathcal A}_\oplus ) \qquad
(C{\mathcal A})_\oplus \cong C({\mathcal A}_\oplus )
\]
\noproof
\end{proposition}

Given a $C^\ast$-category $\mathcal A$, we would like to define an
associated $C^\ast$-algebra that, roughly speaking, carries the same
information as the $C^\ast$-category $\mathcal A$.  The naive way to
do this is to simply form the completion of the union
$$\bigcup_{a_i \in \Ob ({\mathcal A}} \Hom (a_1 \oplus \cdots \oplus a_n,a_1 \oplus \cdots \oplus a_n)_{{\mathcal A}_\oplus}$$
with respect to the inclusions
$$\Hom (a\oplus c,a\oplus c)_{{\mathcal A}_\oplus} \rightarrow \Hom (a\oplus b\oplus c, a\oplus b\oplus c)_{{\mathcal A}_\oplus}$$
\[
\left( \begin{array}{cc}
w & x \\
y & z \\
\end{array} \right)
\mapsto
\left( \begin{array}{ccc}
w & 0 & x \\
0 & 0 & 0 \\
y & 0 & z \\
\end{array} \right)
\]
Unfortunately, this construction is not functorial.  We can, however, replace it by an equivalent functorial construction.

\begin{definition}
Let $\mathcal A$ be a $C^\ast$-category.  Then we define ${\mathcal
O}_{\mathcal A}$ to be the category in which the objects are all
compositions of inclusions of the above form.

A morphism set between two inclusions has precisely one element if the inclusions are composable; otherwise, it is empty.
\end{definition}

We define a functor, $H_{\mathcal A}$, from the category ${\mathcal
O}_{\mathcal A}$ to the category of $C^\ast$-algebras by associating
the $C^\ast$-algebra $\Hom (a\oplus c,a\oplus c)_{{\mathcal
A}_\oplus}$ to the inclusion $\Hom (a\oplus c,a\oplus c)_{{\mathcal
A}_\oplus} \rightarrow \Hom (a\oplus b\oplus c,a\oplus b\oplus
c)_{{\mathcal A}_\oplus}$.  If $i$ and $j$ are composable
inclusions, then the one morphism in the set $\Hom (i,j)_{{\mathcal
O}_{\mathcal A}}$ is mapped to the inclusion $i$ itself.

It is a well-known fact (see for example appendix L of \cite{W-O})
that the category of $C^\ast$-algebras is closed under the formation
of direct limits.  The following definition therefore makes sense.

\begin{definition}
Let $\mathcal A$ be a $C^\ast$-category.  Then we define the $C^\ast$-algebra
${\mathcal A}^H$ to be the colimit of the functor $H_{\mathcal A}$.
\end{definition}

By construction, the assignment ${\mathcal A}\mapsto {\mathcal A}^H$
is a functor from the category of $C^\ast$-categories and
$\ast$-functors to the category of $C^\ast$-algebras.  There is an
obvious natural faithful $\ast$-functor $H\colon {\mathcal A}_\oplus
\rightarrow {\mathcal A}^H$.

\begin{proposition} \label{AHCA}
Let $A$ be a $C^\ast$-algebra.  Then the $C^\ast$-algebra $A^H$ is
naturally isomorphic to the tensor product $A\otimes {\mathcal K}$,
where $\mathcal K$ is the $C^\ast$-algebra of compact operators on a
separable Hilbert space.
\end{proposition}

\begin{proof}
Let ${\mathcal O}'_A$ be the full subcategory of the category
${\mathcal O}_A$ in which the set of objects consists of all
inclusions of the form
\[
x\mapsto \left( \begin{array}{cc}
x & 0 \\
0 & 0 \\
\end{array} \right)
\]

Then the restriction of the functor $H_A$ to the category ${\mathcal
O}'_A$ has colimit $A\otimes {\mathcal K}$ (see for example
\cite{W-O}, appendix L).

But the categories ${\mathcal O}_A$ and ${\mathcal O}'_A$ are
directed systems, and the category ${\mathcal O}'_A$ is cofinal in
the category ${\mathcal O}_A$.  The result now follows.
\end{proof}

\begin{corollary} \label{AHCAC}
Let $\mathcal A$ be a $C^\ast$-category.  Then the $C^\ast$-algebras ${\mathcal A}^H$ and $({\mathcal A}^H)^H$ are naturally isomorphic.
\end{corollary}

\begin{proof}
By the above proposition, it suffices to show that the
$C^\ast$-algebra ${\mathcal A}^H$ is stable, that is to say there is
a natural isomorphism ${\mathcal A}^H \cong {\mathcal A}^H\otimes
{\mathcal K}$.

The tensor product ${\mathcal A}^H\otimes {\mathcal K}$ can be
viewed as the direct limit of the sequence of matrix
$C^\ast$-algebras, $M_n({\mathcal A}^H)$, under the inclusions
\[
x\mapsto \left( \begin{array}{cc}
x & 0 \\
0 & 0 \\
\end{array} \right)
\]

By the universal property of direct limits and definition of the
$C^\ast$-algebra ${\mathcal A}^H$, the above $C^\ast$-algebra is
isomorphic to the the $C^\ast$-algebra ${\mathcal A}^H$, and we are
done.
\end{proof}

The following result is easy to check.

\begin{proposition} \label{Hcom}
Let $\mathcal A$ be a $C^\ast$-category.  Then we have natural isomorphisms
\[
(\Sigma {\mathcal A})^H\cong \Sigma ({\mathcal A}^H) \qquad (J{\mathcal A})^H \cong J({\mathcal A}^H) \qquad (C{\mathcal A})^H\cong C({\mathcal A}^H)
\]
\noproof
\end{proposition}

Because of the above result, we will not worry about distinguishing between the $C^\ast$-categories $(\Sigma {\mathcal A})^H$ and $\Sigma ({\mathcal A}^H)$, or between the $C^\ast$-categories $(J {\mathcal A})^H$ and $J ({\mathcal A}^H)$.

\section{The $KK$-theory spectrum}

Recall that at the most basic level, a {\em spectrum}, $\mathbb E$,
is a sequence of topological spaces, $E_n$, each of which is
equipped with a basepoint, together with continuous maps $\epsilon
\colon E_n \rightarrow \Omega E_{n+1}$.  The book \cite{Ad1} can be
consulted for further details.

\begin{definition}
Let $\mathcal A$ and $\mathcal B$ be $C^\ast$-categories. Then we
define $F({\mathcal A},{\mathcal B})$ to be the space of all
$\ast$-functors ${\mathcal A}\rightarrow {\mathcal B}^H$. This
function space is equipped with the compact open topology.
\end{definition}

The above metric is well-defined, since any $\ast$-functor between $C^\ast$-categories is norm-decreasing.

Let $\mathcal A$ be a fixed $C^\ast$-category.  Then the assignment
${\mathcal B}\mapsto F({\mathcal A},{\mathcal B})$ is a covariant
functor from the category of $C^\ast$-categories and $\ast$-functors
to the category of topological spaces with basepoint.  Given a fixed
$C^\ast$-category $\mathcal B$, the assignment ${\mathcal A}\mapsto
F({\mathcal A},{\mathcal B})$ is a contravariant functor from the
category of $C^\ast$-categories to the category of topological
spaces with base-point.

Let $k\in \N$.  Define a $C^\ast$-category $J^k {\mathcal A}$ by
writing
$$J^0 {\mathcal A} ={\mathcal A} \qquad J^{k+1}{\mathcal A} = J(J^k
{\mathcal A})$$ using the functor $J$ from section \ref{prelim}.

Consider a $\ast$-functor $\alpha \colon J^k {\mathcal A}\rightarrow
{\mathcal B}^H$.  Then by functoriality of the construction
${\mathcal B}\mapsto {\mathcal B}^H$ and proposition \ref{Hcom} we
have a semi-split exact sequence
\[
0\rightarrow \Sigma {\mathcal B}^H \rightarrow C{\mathcal B}^H \rightarrow {\mathcal B}^H \rightarrow 0
\]

Therefore, by proposition \ref{classmap}, we have a functorial
classifying map
\[
\eta (\alpha ) \colon J^{k+1} {\mathcal A}\rightarrow \Sigma {\mathcal B}^H
\]

\begin{definition}
We define $\KK({\mathcal A},{\mathcal B})$ to be the spectrum with
sequence of spaces $(F(J^{2n} {\mathcal A},\Sigma^n {\mathcal B}))$.
The structure map
\[
\epsilon \colon F(J^{2n} {\mathcal A},\Sigma^n {\mathcal B})
\rightarrow \Omega F(J^{2n+2}{\mathcal A},\Sigma^{n+1}{\mathcal B})
\cong F(J^{2n+2}{\mathcal A},\Sigma^{n+2}{\mathcal B})
\]
is defined by applying the above classifying map construction twice,
that is to say writing $\epsilon (\alpha ) = \eta (\eta (\alpha ))$
whenever $\alpha \in F(J^{2n}{\mathcal A},\Sigma^n {\mathcal B})$.
\end{definition}

We would like to be able to define certain products at the level of
spectra.  In order to do this, we need to have some extra structure.

The following definition comes from \cite{HSS}.

\begin{definition}
A spectrum, $\mathbb E$, is called a {\em symmetric spectrum} if
each space $E_n$ is equipped with an action of the symmetric group
$S_n$, and the iterated structure map $\epsilon^k \colon E_n
\rightarrow \Omega^k E_{n+k}$ is $S_n\times S_k$-equivariant in the
obvious sense.
\end{definition}

\begin{proposition}
The spectrum $\KK({\mathcal A},{\mathcal B})$ is a symmetric spectrum.
\end{proposition}

\begin{proof}
The $C^\ast$-category $\Sigma^n{\mathcal B}$ can be viewed as the
tensor product $C_0(0,1)\otimes \cdots \otimes C_0(0,1)\otimes
{\mathcal B}$, where there are $n$ copies of the $C^\ast$-algebra
$C_0 (0,1)$.  There is therefore a canonical action of the
permutation group $S_n$ on the space $(F(J^{2n} {\mathcal
A},\Sigma^n {\mathcal B}))$ defined by permuting the copies of the
$C^\ast$-algebra $C_0 (0,1)$.

By naturality of the classifying map construction, the iterated
structure map $\epsilon^k \colon F(J^{2n}{\mathcal
A},\Sigma^n{\mathcal B}) \rightarrow \Omega^k F(J^{2n+2k}{\mathcal
A},\Sigma^{n+k}{\mathcal B})$ is $S_n \times S_k$-equivariant, and
so we have a symmetric spectrum as required.
\end{proof}

Let $\mathbb E$, $\mathbb F$, and $\mathbb G$ be symmetric spectra,
with spaces $E_n$, $F_n$, and $G_n$ respectively.  Then there is a
notion of a smash product of symmetric spectra ${\mathbb E}\wedge
{\mathbb F}$.  A collection of continuous basepoint-preserving $S_m
\times S_n$-equivariant maps $E_m \wedge F_n \rightarrow G_{m+n}$
which commute with the structure maps of the spectra define a map of
spectra ${\mathbb E}\wedge {\mathbb F}\rightarrow {\mathbb G}$.

\begin{proposition} \label{sdef}
Let $\mathcal A$ be a $C^\ast$-category, and let $k$ and $l$ be
natural numbers.  Then there is a natural $\ast$-functor $s\colon
J^k\Sigma^l {\mathcal A}\rightarrow \Sigma^l J^k{\mathcal A}$.
\end{proposition}

\begin{proof}
The classifying map of the diagram
\[
\begin{array}{ccccccccc}
& & & & & & \Sigma {\mathcal A} \\
& & & & & & \| \\
0 & \rightarrow & \Sigma J{\mathcal A} & \rightarrow & \Sigma T {\mathcal A} & \rightarrow & \Sigma {\mathcal A} & \rightarrow & 0 \\
\end{array}
\]
is a natural $\ast$-functor $J\Sigma {\mathcal A}\rightarrow \Sigma
J{\mathcal A}$.  The $\ast$-functor $s$ is defined by iterating the
above construction.
\end{proof}

The following definition now makes sense by proposition \ref{Hcom} and corollary \ref{AHCAC}.

\begin{definition} \label{Shprod}
Let $\mathcal A$, $\mathcal B$, and $\mathcal C$ be
$C^\ast$-categories.  Let $\alpha \in F(J^{2m} {\mathcal A},
\Sigma^m {\mathcal B})$ and $\beta \in F(J^{2n} {\mathcal
B},\Sigma^n {\mathcal C})$.  Then we define the product $\alpha
\sharp \beta$ to be the composition
\[
J^{2m+2n}{\mathcal A}\stackrel{J^{2n} \alpha}{\rightarrow}
J^{2n}\Sigma^m {\mathcal B}^H \stackrel{s}{\rightarrow} \Sigma^m
J^{2n}{\mathcal B}^H \stackrel{\Sigma^m \beta^H}{\rightarrow}
\Sigma^{m+n}(({\mathcal C}^H)^H) \cong \Sigma^{m+n}({\mathcal C}^H)
\]
\end{definition}

\begin{theorem} \label{KKprod}
Let $\mathcal A$, $\mathcal B$, and $\mathcal C$ be $C^\ast$-categories.  Then there is a natural map of spectra
\[
\KK({\mathcal A},{\mathcal B}) \wedge \KK({\mathcal B},{\mathcal C}) \rightarrow \KK({\mathcal A},{\mathcal C})
\]
defined by the formula
\[
\alpha \wedge \beta \mapsto \alpha \sharp \beta \qquad \alpha \in F(J^m {\mathcal A}, {\mathcal B}),\ \beta \in F(J^n {\mathcal B},{\mathcal C})
\]

Further, the above product is associative.  To be precise, let
$\alpha \in F(J^m {\mathcal A},{\mathcal B})$, $\beta \in F(J^n
{\mathcal B},{\mathcal C})$, and $\gamma \in F(J^p {\mathcal
C},{\mathcal D})$.  Then $(\alpha \sharp \beta )\sharp \gamma =
\alpha \sharp (\beta \sharp \gamma )$.
\end{theorem}

\begin{proof}
Our construction gives us a natural continuous $S_m \times
S_n$-equivariant map $F(J^{2m}{\mathcal A},\Sigma^m{\mathcal B})
\wedge F(J^{2n} {\mathcal B}, \Sigma^n {\mathcal C}) \rightarrow
F(J^{2m+2n}{\mathcal A},\Sigma^{m+n}{\mathcal C})$.  Compatibility
with the structure maps follows since naturality of the classifying
map construction gives us a commutative diagram
\[
\begin{array}{ccccccc}
J^{2m+2n+2}{\mathcal A} & \stackrel{J^{2n+2}(\alpha )}{\rightarrow} & J^{2n+2}\Sigma^m {\mathcal B}^H & \stackrel{s}{\rightarrow} & \Sigma^m J^{2n+2} {\mathcal B}^H & \stackrel{\Sigma^m \eta^2 (\beta^H )}{\rightarrow} & \Sigma^{m+n+2} {\mathcal C}^H \\
\| & & \downarrow & & \downarrow & & \| \\
J^{2m+2n+2}{\mathcal A} & \stackrel{J^{2n} \eta^2(\alpha )}{\rightarrow} & J^{2n}\Sigma^{m+2} {\mathcal B}^H & \stackrel{s}{\rightarrow} & \Sigma^{m+2} J^{2n} {\mathcal B}^H & \stackrel{\Sigma^{m+2}\beta_\oplus }{\rightarrow} & \Sigma^{m+n+2} {\mathcal C}^H \\
\end{array}
\]
where the non-trivial vertical maps come from iterating the classifying map of the diagram
\[
\begin{array}{ccccccccc}
& & & & & & {\mathcal B}^H \\
& & & & & & \| \\
0 & \rightarrow & \Sigma {\mathcal B}^H & \rightarrow & C{\mathcal B}^H & \rightarrow & {\mathcal B}^H & \rightarrow & 0 \\
\end{array}
\]

We now need to check the statement concerning associativity.  Consider $\ast$-functors
\[
\alpha \colon J^{2m}{\mathcal A}\rightarrow \Sigma^m {\mathcal B}^H
\qquad
\beta \colon J^{2n}{\mathcal B}\rightarrow \Sigma^n {\mathcal C}^H
\qquad
\gamma \colon J^{2p}{\mathcal C}\rightarrow \Sigma^p {\mathcal D}^H
\]

Then we have a commutative diagram
\[
\begin{array}{ccc}
J^{2m+2n+2p}{\mathcal A} & = & J^{2m+2n+2p}{\mathcal A} \\
\downarrow & & \downarrow \\
J^{2n+2l}\Sigma^m {\mathcal B}^H & = & J^{2n+2p}\Sigma^m {\mathcal B}^H \\
\downarrow & & \downarrow \\
J^{2p}\Sigma^m J^{2n}{\mathcal B}^H & \stackrel{s}{\rightarrow} & \Sigma^m J^{2n+2p} {\mathcal B}^H \\
\downarrow & & \downarrow \\
J^{2p} \Sigma^{m+n} {\mathcal C}^H & \stackrel{s}{\rightarrow} & \Sigma^m J^{2p} \Sigma^n {\mathcal C}^H \\
\downarrow & & \downarrow \\
\Sigma^{m+n}J^{2p} {\mathcal D}^H & = & \Sigma^{m+n}J^{2p} {\mathcal D}^H \\
\downarrow & & \downarrow \\
\Sigma^{m+n+p} {\mathcal D}^H & = & \Sigma^{m+n+p} {\mathcal D}^H \\
\end{array}
\]

But the column on the left is the product $(\alpha \sharp \beta
)\sharp \gamma$ and the column on the right is the product $\alpha
\sharp (\beta \sharp \gamma)$ so associativity of the product
follows.
\end{proof}

By definition of our product, the following result holds.

\begin{proposition} \label{pco}
Let $\alpha \colon {\mathcal A}\rightarrow {\mathcal B}$ and $\beta
\colon {\mathcal B} \rightarrow {\mathcal C}$ be $\ast$-functors.
Then $\alpha \sharp \beta = \beta \circ \alpha$. \noproof
\end{proposition}

\begin{proposition} \label{dmap}
Let $\mathcal A$, $\mathcal B$, and $\mathcal C$ be
$C^\ast$-categories.  Then there is a map $\Delta \colon \KK
({\mathcal A},{\mathcal B})\rightarrow \KK ({\mathcal A}\otimes
{\mathcal C},{\mathcal B}\otimes {\mathcal C})$.  This map is
compatible with the product in the sense that we have a commutative
diagram
\[
\begin{array}{ccc}
\KK ({\mathcal A},{\mathcal B})\wedge \KK ({\mathcal B},{\mathcal C}) & \rightarrow & \KK ({\mathcal A},{\mathcal C}) \\
\downarrow & & \downarrow \\
\KK ({\mathcal A}\otimes {\mathcal D},{\mathcal B}\otimes {\mathcal D})\wedge \KK ({\mathcal B}\otimes {\mathcal D},{\mathcal C}\otimes {\mathcal D}) & \rightarrow & \KK ({\mathcal A}\otimes {\mathcal D},{\mathcal C}\otimes {\mathcal D}) \\
\end{array}
\]
where the horizontal maps are defined by the product, and the vertical maps are copies of the map $\Delta$.
\end{proposition}

\begin{proof}
Let $\alpha \colon J^{2n}{\mathcal A} \rightarrow \Sigma^n{\mathcal B}^H$ be a $\ast$-functor.  Then, since $\Sigma {\mathcal B}^H = C_0 (0,1)\otimes {\mathcal B}^H$, we have a naturally induced $\ast$-functor $\alpha \otimes 1 \colon (J^{2n}{\mathcal A})\otimes {\mathcal C}\rightarrow \Sigma^n({\mathcal B}^H \otimes {\mathcal C})$.

There is an obvious natural $\ast$-functor $\beta \colon {\mathcal
B}^H \otimes {\mathcal C} \rightarrow ({\mathcal B}\otimes {\mathcal
C})^H$.  By corollary \ref{jtens}, there is a short exact sequence
$$0 \rightarrow (J{\mathcal A})\otimes {\mathcal C} \rightarrow
(T{\mathcal A})\otimes {\mathcal C} \rightarrow {\mathcal A}\otimes
{\mathcal C} \rightarrow 0$$ with a completely positive splitting.
We thus obtain a natural $\ast$-functor $\gamma \colon J({\mathcal
A}\otimes {\mathcal C})\rightarrow (J{\mathcal A})\otimes {\mathcal
C}$ as the classifying map of the diagram
\[
\begin{array}{ccccccccc}
& & & & & & {\mathcal A}\otimes {\mathcal C} \\
& & & & & & \| \\
0 & \rightarrow & (J{\mathcal A})\otimes {\mathcal C} & \rightarrow & (T{\mathcal A})\otimes {\mathcal C} & \rightarrow & {\mathcal A}\otimes {\mathcal C} & \rightarrow & 0 \\
\end{array}
\]

We now define the map $\Delta$ by writing $\Delta (\alpha ) = \beta
\circ (\alpha \otimes 1)\circ \gamma^n$.  Compatibility with the
product is easy to check.
\end{proof}

\section{Ring and Module Structure}

A {\em symmetric monoidal category} is a category with a sensible
idea of a product of objects, $\wedge$, along with a unit object $e$
equipped with isomorphisms $e\wedge X\rightarrow X$ and $X\wedge e
\rightarrow X$ for any object $X$ in the category.  Any standard
book on category theory, for example \cite{Mac}, can be consulted
for further details.  It is proven in \cite{HSS} that the category
of symmetric spectra is a symmetric monoidal category.  The product
is the smash product of spectra.  The unit is the sequence of spaces
\[
e= (S^0 , \star ,\star ,\ldots )
\]
where $\star$ is the one point topological space, and $S^0$ is the disjoint union of two points.  By definition of the smash product in the category of symmetric spectra, there is a unique natural isomorphism between the objects $e\wedge {\mathbb E}$, ${\mathbb E}\wedge e$, and $\mathbb E$ for any spectrum $\mathbb E$.

\begin{definition}
A {\em symmetric ring spectrum} is a monoid in the category of symmetric spectra.
\end{definition}

To be more precise, a symmetric spectrum $R$ is called a {\em
symmetric ring spectrum} if it is equipped with an associative
product $\mu \colon R\wedge R\rightarrow R$ and a unit map $\eta
\colon e\rightarrow R$ such that we have a commutative diagram
\[
\begin{array}{ccccc}
e\wedge R & \stackrel{\eta \wedge 1}{\rightarrow} & R\wedge R & \stackrel{1\wedge \eta}{\leftarrow} & R\wedge e \\
\downarrow & & \downarrow & & \downarrow \\
R & = & R & = & R \\
\end{array}
\]

Here the central vertical map is the product $\mu$.  The vertical
maps on the left and right are the isomorphisms associated with the
unit $e$.

\begin{theorem} \label{RS}
Let $\mathcal A$ be a $C^\ast$-category.  Then the spectrum $\KK ({\mathcal A},{\mathcal A})$ is a symmetric ring spectrum.
\end{theorem}

\begin{proof}
By theorem \ref{KKprod} we have an associative product
\[
\mu \colon \KK ({\mathcal A},{\mathcal A})\wedge \KK ({\mathcal A},{\mathcal A}) \rightarrow \KK ({\mathcal A},{\mathcal A})
\]

Recall that the unit, $e$, is the sequence of spaces $(S^0 , \star ,
\star ,\ldots )$.  Thus there is only one point in the $0$-th space
that is not a basepoint.  We can define a unit map $\eta \colon
e\rightarrow \KK ({\mathcal A},{\mathcal A})$ by mapping the base
point of the $n$-th space of the spectrum $e$ to the base point of
the $n$-th space of the spectrum $\KK ({\mathcal A},{\mathcal A})$,
and mapping the point in $S^0$ which is not a basepoint to the point
in the space $F({\mathcal A},{\mathcal A})$ arising from the
identity $\ast$-functor $1\colon {\mathcal A}\rightarrow {\mathcal
A}$.

Commutativity of the diagram involving the unit follows from proposition \ref{pco}.
\end{proof}

The following definition comes from applying a definition for
symmetric monoidal categories to the category of symmetric spectra.

\begin{definition}
Let $R$ be a symmetric ring spectrum equipped with multiplication
$\mu$.  Then a symmetric spectrum $M$ is called a {\em symmetric
$($left$)$ $R$-module spectrum} if it comes equipped with a
multiplication $\mu'\colon R\wedge M \rightarrow M$ such that we
have a commutative diagram
\[
\begin{array}{ccc}
R\wedge R\wedge M & \stackrel{\mu \wedge 1}{\rightarrow} & R\wedge M \\
\downarrow & & \downarrow \\
R\wedge M & \stackrel{\mu'}{\rightarrow} & M \\
\end{array}
\]

Here the vertical map on the left is the product $1\wedge \mu'$ and the vertical map on the right is the product $\mu'$.
\end{definition}

\begin{theorem} \label{RS2}
Let $\mathcal A$ and $\mathcal B$ be $C^\ast$-categories.  Then the
spectrum $\KK ({\mathcal A},{\mathcal B})$ is a symmetric $\KK (\F
,\F )$-module spectrum.
\end{theorem}

\begin{proof}
By theorem \ref{KKprod} and proposition \ref{dmap} we can define a suitable product
\[
\KK (\F , \F) \wedge \KK ({\mathcal A},{\mathcal B})
\stackrel{\Delta \wedge 1}{\rightarrow} \KK ({\mathcal A},{\mathcal
A}) \wedge \KK ({\mathcal A},{\mathcal B})
\stackrel{\mu}{\rightarrow} \KK ({\mathcal A},{\mathcal B})
\]
\end{proof}

\section{The Equivariant Case}

Let $\mathcal G$ be a discrete groupoid.  We will regard $\mathcal
G$ as a small category in which every morphism is invertible. Taking
this point of view, we define a {\em $\mathcal G$-algebra} to be a
functor from the category $\mathcal G$ to the category of algebras
and homomorphisms.  Similarly (see \cite{Mitch6}) a {\em $\mathcal
G$}-$C^\ast$-algebra is a functor from the category $\mathcal G$ to
the category of $C^\ast$-algebras and $\ast$-homomorphisms.

Thus, if $A$ is a $\mathcal G$-$C^\ast$-algebra, then for each
object $a\in \Ob ({\mathcal G})$ we have a $C^\ast$-algebra $A(a)$.
A morphism $g\in \Hom (a,b)_{\mathcal G}$ induces a homomorphism
$g\colon A(a)\rightarrow A(b)$.

We can regard an ordinary $C^\ast$-algebra $C$ as a $\mathcal
G$-$C^\ast$-algebra by writing $C(a)=C$ for each object $a\in \Ob
({\mathcal G})$ and saying that each morphism in the groupoid
$\mathcal G$ acts as the identity map.

A {\em $\mathcal G$-equivariant map} (or more simply {\em
equivariant map}, when we do not need to emphasize the groupoid
$\mathcal G$) between $\mathcal G$-$C^\ast$-algebras $A$ and $B$ is
a natural transformation from the functor $A$ to the functor $B$. We
can similarly talk about equivariant completely positive maps.

A {\em short exact sequence} of $\mathcal G$-$C^\ast$-algebras is a
sequence of $\mathcal G$-$C^\ast$-algebras and equivariant maps
\[
A\stackrel{f}{\rightarrow} B \stackrel{g}{\rightarrow} C
\]
such that the sequence
\[
0\rightarrow A(a)\stackrel{f}{\rightarrow} B(a)\stackrel{g}{\rightarrow} C(a)\rightarrow 0
\]
is exact for each object $a\in \Ob ({\mathcal G})$.  A splitting of
a short exact sequence is defined in the obvious way, as is a
completely positive splitting.  As before, we refer to split exact
sequences and semi-split exact sequences respectively.

To define such a $KK$-theory spectrum for $\mathcal
G$-$C^\ast$-algebras, we need variations of the various
constructions defined earlier for $C^\ast$-categories.

\begin{definition}
Let $\mathcal G$ be a discrete groupoid, and let $A$ be a $\mathcal
G$-$C^\ast$-algebra.  Then we define $A[0,1]$ to be the $\mathcal
G$-$C^\ast$-algebra where the algebra $A[0,1](a)$ consists of all
continuous functions $f\colon [0,1] \rightarrow \Hom A(a)$.  The
$\mathcal G$-action is defined by the formula
\[
g(f)(t) = g(f(t)) \qquad g\in \Hom (a,b)_{\mathcal G},\ f\colon [0,1]\rightarrow A(a),\ t\in [0,1]
\]

We define the {\em cone}, $CA$ to be the $\mathcal G$-$C^\ast$-algebra where
\[
CA(a) = \{ f\in A[0,1](a)\ |\ f(0)=0 \} \qquad a\in \Ob ({\mathcal G})
\]
and the $\mathcal G$-action is as defined above.  The {\em suspension}, $\Sigma {A}$ is defined similarly by writing
\[
\Sigma A(a) = \{ f\in CA(a)\ |\ f(1)=0 \} \qquad a\in \Ob ({\mathcal G})
\]
\end{definition}

There is an obvious natural equivariant map $i\colon \Sigma
A\rightarrow CA$ defined by inclusion.  There is also an equivariant
map $j\colon CA\rightarrow A$, defined by the formula $j(f)=f(1)$,
where $f\in CA(a)$.  It is easy to check that we have a short exact
sequence
\[
0 \rightarrow \Sigma {A} \rightarrow C{A}\rightarrow {A}\rightarrow 0
\]

The above exact sequence has a natural completely positive splitting $s\colon {A}\rightarrow C{A}$ defined by the formula
\[
s (x)(t)= tx \qquad t\in [0,1],\ x\in A(a)
\]

\begin{definition}
Let $A$ and $B$ be $\mathcal G$-$C^\ast$-algebras.  Then we define
the {\em tensor product} $A\otimes B$ to be the $\mathcal
G$-$C^\ast$-algebra where $(A\otimes B)(a) = A(a)\otimes B(a)$, and
the $\mathcal G$-action is defined by writing $g(x\otimes y) =
g(x)\otimes g(y)$ whenever $g\in \Hom (a,b)_{\mathcal G}$, $x\in
A(a)$, and $y\in B(a)$.

We define the {\em algebraic tensor product} $A\odot B$ to be the $\mathcal G$-algebra where $(A\odot B)(a) = A(a)\odot B(a)$,
and the $\mathcal G$-action is defined as above.

We define the {\em direct sum} $A\oplus B$ to be the $\mathcal
G$-$C^\ast$-algebra where $(A\oplus B)(a)$ is the direct sum
$A(a)\oplus B(a)$ for each object $a\in \Ob ({\mathcal G})$ and the
$\mathcal G$-action is defined by writing $g(x\oplus y)= g(x)\oplus
g(y)$ whenever $g\in \Hom (a,b)_{\mathcal G}$, $x\in A(a)$, and
$y\in B(a)$.
\end{definition}

The above tensor product of $C^\ast$-algebras can be assumed to be the spatial tensor product.

\begin{definition}
Let $A$ be a $\mathcal G$-$C^\ast$-algebra.  Then we define
$A^{\odot k}$ to be the algebraic tensor product of $A$ with itself
$k$ times.  We define the {\em equivariant algebraic tensor
algebra}, $T_\mathrm{alg}A$, to be the completion of the iterated
direct sum
\[
T_\mathrm{alg}A = \oplus_{k=1}^\infty A^{\odot k}
\]

Multiplication in the algebra $T_\mathrm{alg}A$ is given by concatenation of tensors.
\end{definition}

We have a canonical set of linear maps $\sigma \colon {A}\rightarrow
T_\mathrm{alg}{A}$ defined by mapping each element of the algebra
$A$ onto the first summand.  The following result is easy to check.

\begin{proposition}
The equivariant tensor algebra, $T_\mathrm{alg}{A}$, can be equipped
with a $C^\ast$-norm given by defining the norm, $\| u \|$, of an
element $u$ in the equivariant tensor algebra to be the supremum of
all values $p( \varphi (u))$ where $\varphi \colon T_\mathrm{alg}
{A}\rightarrow {B}$ is an equivariant completely positive map into a
$\mathcal G$-$C^\ast$-algebra $B$ such that the composition $\varphi
\circ \sigma \colon {A}\rightarrow {B}$ is completely positive and
norm-decreasing, and $p$ is a $C^\ast$-seminorm on the
$C^\ast$-algebra $B$. \noproof
\end{proposition}

We define the {\em equivariant tensor $C^\ast$-algeba}, $T{\mathcal
A}$, to be the completion of the equivariant tensor algebra
$T_\mathrm{alg}{\mathcal A}$ with respect to the above norm.
Formation of the equivariant tensor algebra defines a functor from
the category of $\mathcal G$-$C^\ast$-algebras and equivariant maps
to itself.  Further, just as we showed in the non-equivariant case,
we have a universal property.

\begin{proposition}
Let $A$ and $B$ be $\mathcal G$-$C^\ast$-algebras.  Let $\alpha
\colon A\rightarrow B$ be a completely positive equivariant map.
Then there is a unique equivariant map $\varphi \colon TA\rightarrow
B$ such that $\alpha = \varphi \circ \sigma$. \noproof
\end{proposition}

There is a natural equivariant map $\pi \colon TA\rightarrow A$ defined by the formula
\[
\varphi (x_1 \otimes \cdots \otimes x_n) = x_n\ldots x_1 \qquad x_i \in A(a),\ a\in \Ob ({\mathcal G})
\]

We can thus define a $\mathcal G$-$C^\ast$-algebra $JA$ by writing
\[
JA(a) = \ker \pi \colon TA(a)\rightarrow A(a)
\]
for each object $a\in \Ob ({\mathcal G})$.  The $\mathcal G$-action
is inherited from the equivariant tensor algebra.  There is a
natural semi-split short exact sequence
\[
0\rightarrow J{A} \hookrightarrow T{A} \stackrel{\pi}{\rightarrow} {A}\rightarrow 0
\]
with completely positive splitting $\sigma \colon {A}\rightarrow
T{A}$.  The following result is proved in the same way as
proposition \ref{classmap}.

\begin{proposition} \label{classmape}
Let
\[
0\rightarrow {I}\rightarrow {E} \rightarrow {B} \rightarrow 0
\]
be a semi-split exact sequence of $\mathcal G$-$C^\ast$-algebras.
Let $\alpha \colon {A}\rightarrow {B}$ be an equivariant map.  Then
there are equivariant maps $\tau \colon T{A}\rightarrow {E}$ and
$\gamma \colon J{A}\rightarrow {I}$ such that we have a commutative
diagram
\[
\begin{array}{ccccccccc}
0 & \rightarrow & J{A} & \rightarrow & T{A} & \rightarrow & {A} & \rightarrow & 0 \\
& & \downarrow & & \downarrow & & \downarrow \\
0 & \rightarrow & {I} & \rightarrow & {E} & \rightarrow & {B} & \rightarrow & 0 \\
\end{array}
\]
\noproof
\end{proposition}

\begin{definition}
The homomorphism $\gamma$ is called the {\em classifying map} of the diagram
\[
\begin{array}{ccccccccc}
& & & & & & {A} \\
& & & & & & \downarrow \\
0 & \rightarrow & {I} & \rightarrow & {E} & \rightarrow & {B} & \rightarrow & 0 \\
\end{array}
\]
\end{definition}

\begin{definition}
Let $A$ and $B$ be $\mathcal G$-$C^\ast$-algebras.  Then we define
$F_{\mathcal G}(A,B)$ to be the space of all equivariant maps
$A\rightarrow B\otimes {\mathcal K}$.    The topology on this
function space is the compact open topology.
\end{definition}

The above metric is well-defined, since any $\ast$-homomorphism
between $C^\ast$-algebras is norm-decreasing.

Consider an equivariant map $\alpha \in F_{\mathcal G}(J^k {A},B)$.
Then we can check (for example, by looking at matrices and then
taking direct limits) that we have a semi-split exact sequence
\[
0\rightarrow \Sigma {B}\otimes {\mathcal K} \rightarrow C{B}\otimes {\mathcal K} \rightarrow {B}\otimes {\mathcal K} \rightarrow 0
\]
and so, by proposition \ref{classmape}, a classifying map
\[
\eta (\alpha ) \colon J^{k+1} {A} \rightarrow \Sigma B\otimes {\mathcal K}
\]

\begin{definition}
We define $\KK_{\mathcal G}({A},{B})$ to be the symmetric spectrum
with sequence of spaces $(F_{\mathcal G}(J^{2n} {A},\Sigma^n{B}))$
with $S_n$-action defined by permuting the order in which the
suspensions are made.  The classifying map
\[
\epsilon \colon F_{\mathcal G}(J^{2n} {A},\Sigma^n {B}) \rightarrow
\Omega F_{\mathcal G}(J^{2n+2}{A},\Sigma^{n+1}{B}) \cong F_{\mathcal
G}(J^{2n+2}{A},\Sigma^{n+2}{B})
\]
is defined by applying the above classifying map construction twice,
that is to say writing $\epsilon (\alpha ) = \eta (\eta (\alpha ))$
whenever $\alpha \in F_{\mathcal G}(J^{2n}{A},\Sigma^n B)$.
\end{definition}

Just as in proposition \ref{sdef}, for any $\mathcal
G$-$C^\ast$-algebra $A$, there is a natural equivariant map $s\colon
J^k\Sigma^l A\rightarrow \Sigma^l J^k A$.

\begin{definition}
Let $A$, $B$, and $C$ be $\mathcal G$-$C^\ast$-algebras.  Let
$\alpha \in F_{\mathcal G}(J^{2m} {A}, \Sigma^m{B})$ and $\beta \in
F_{\mathcal G}(J^{2n} {B},\Sigma^n{C})$.  Then we define the product
$\alpha \sharp \beta$ to be the composition
\[
J^{2m+2n}{A} \stackrel{J^{2n} \alpha}{\rightarrow} J^{2n} \Sigma^m
(B\otimes {\mathcal K}) \stackrel{s}{\rightarrow} \Sigma^m J^{2n}
(B\otimes {\mathcal K})\stackrel{\Sigma^m(\beta \otimes
1)}{\rightarrow} \Sigma^{m+n}C\otimes {\mathcal K}
\]
\end{definition}

The following result is proved in exactly the same way as theorem \ref{KKprod}.

\begin{theorem} \label{KKprod2}
Let $A$, $B$, and $C$ be $\mathcal G$-$C^\ast$-algebras.  Then there is a natural map of spectra
\[
\KK_{\mathcal G}({A},{B}) \wedge \KK_{\mathcal G}(B,{C}) \rightarrow \KK_{\mathcal G}({A},{C})
\]
defined by the formula
\[
\alpha \wedge \beta \mapsto \alpha \sharp \beta \qquad \alpha \in F^0(J^m {A}, {B}),\ \beta \in F^0(J^n {B},{C})
\]

Further, the above product is associative.  To be precise, let
$\alpha \in F_{\mathcal G}(J^m {A},{B})$, $\beta \in F_{\mathcal
G}(J^n {B},{C})$, and $\gamma \in F_{\mathcal G}(J^p {C},{D})$. Then
$(\alpha \sharp \beta )\sharp \gamma = \alpha \sharp (\beta \sharp
\gamma )$. \noproof
\end{theorem}

As before, the following result is obvious.

\begin{proposition}
Let $\alpha \colon A\rightarrow B$ and $\beta \colon B \rightarrow C$ be equivariant maps.  Then $\alpha \sharp \beta = \beta \circ \alpha$.
\noproof
\end{proposition}

The following result is proved in exactly the same way as theorems \ref{RS} and \ref{RS2}.

\begin{theorem}
Let $\mathcal G$ be a discrete groupoid, and let $A$ be a $\mathcal
G$-$C^\ast$-algebra.  Then the spectrum $\KK_{\mathcal G} ({A},{A})$
is a symmetric ring spectrum.

Let $B$ be another $\mathcal G$-$C^\ast$-algebra.  Then the spectrum
$\KK_{\mathcal G}(A,B)$ is a symmetric $\KK_{\mathcal G}(\F ,\F
)$-module spectrum. \noproof
\end{theorem}

Let $\theta \colon {\mathcal G}\rightarrow {\mathcal H}$ be a
functor between groupoids, and let $A$ be an $\mathcal
H$-$C^\ast$-algebra.  Abusing notation, we can also regard $A$ as a
$\mathcal G$-$C^\ast$-algebra; we write $A(a)=A(\theta (a))$ for
each object $a\in \Ob( {\mathcal G})$, and define a homomorphism
$g=\theta (g) \colon A(\theta (a))\rightarrow A(\theta (b))$ for
each morphism $g\in \Hom (a,b)_{\mathcal G}$.

If $A$ and $B$ are $\mathcal H$-$C^\ast$-algebras, then we have an
induced map $\theta^\ast \colon F_{\mathcal H}(A,B)\rightarrow
F_{\mathcal G}(A,B)$ defined by the observation that any $\mathcal
H$-equivariant map between $A$ and $B\otimes {\mathcal K}$ is also
$\mathcal G$-equivariant.  This induced map is natural in the
variables $A$ and $B$.  Going slightly further, we have the
following easy to check result.

\begin{proposition} \label{restriction}
Let $\theta \colon {\mathcal G}\rightarrow {\mathcal H}$ be a
functor between groupoids, and let $A$ and $B$ be $\mathcal
H$-$C^\ast$-algebras.  Then there is an induced map of spectra
$\theta^\ast \colon \KK_{\mathcal H}(A,B)\rightarrow \KK_{\mathcal
G}(A,B)$.  This induced map is compatible with the product in the
sense that we have a commutative diagram
\[
\begin{array}{ccc}
\KK_{\mathcal H}(A,B)\wedge \KK_{\mathcal H}(B,C) & \rightarrow & \KK_{\mathcal H}(A,C) \\
\downarrow & & \downarrow \\
\KK_{\mathcal G}(A,B)\wedge \KK_{\mathcal G}(B,C) & \rightarrow & \KK_{\mathcal G}(A,C) \\
\end{array}
\]
where the horizontal map is defined by the product and the vertical maps are restriction maps.
\noproof
\end{proposition}

The above map $f^\ast$ is called the {\em restriction map}.

\section{Descent}

Apart from the last theorem, the definitions and results in this section come from \cite{Mitch6}.

Let $\mathcal G$ be a discrete groupoid, and let $A$ be a $\mathcal
G$-$C^\ast$-algebra.  Then we define the {\em convolution algebroid}
$A{\mathcal G}$ to be the algebroid with the same set of objects as
the groupoid $\mathcal G$, and morphism sets
\[
\Hom (a,b)_{A{\mathcal G}} = \{ \sum_{i=1}^m x_i g_i \ |\ x_i \in A(b), g_i \in \Hom (a,b)_{\mathcal G},\ m\in \N \}
\]

Composition of morphisms is defined by the formula
\[
\left( \sum_{i=1}^m x_i g_i \right) \left( \sum_{j=1}^n y_j h_i \right) = \sum_{i,j=1}^{m,n} x_i g_i(y_j) g_i h_j
\]

Further, we have an involution
\[
\left( \sum_{i=1}^m x_i g_i \right)^\ast = \sum_{i=1}^m g_i^{-1} (x_i^\ast )g_i^{-1}
\]

Recall that we write $\mathcal L$ to denote the category of all
Hilbert spaces and bounded linear maps.  We write ${\mathcal L}(H)$
to denote the $C^\ast$-algebra of all bounded linear maps on a
Hilbert space $H$.

\begin{definition}
Let $\mathcal G$ be a discrete groupoid.  Then a {\em unitary
representation} of $\mathcal G$ is a functor $\rho \colon {\mathcal
G}\rightarrow {\mathcal L}$ such that $\rho (g^{-1}) =\rho (g)^\ast$
for each morphism $g$ in the groupoid $\mathcal G$.

Let $A$ be a $\mathcal G$-$C^\ast$-algebra.  Then a {\em covariant
representation of $A$} is a pair $(\rho , \pi )$, where $\rho$ is a
unitary representation of the groupoid $\mathcal G$, and $\pi$ is a
set of representations $\pi \colon A(a) \rightarrow {\mathcal
L}(\rho (a))$, where $a\in \Ob ({\mathcal G})$, such that
\[
\rho (g) \pi (x) = \pi (gx) \rho (g)
\]
for each element $x\in A(a)$ and morphism $g\in \Hom (a,b)_{\mathcal G}$
\end{definition}

Given a covariant representation $(\rho , \pi )$, we have a
$C^\ast$-functor $(\rho , \pi)_\ast \colon A{\mathcal G} \rightarrow
{\mathcal L}$ defined by mapping the object $a\in \Ob ({\mathcal
G})$ to the Hilbert space $\rho (a)$, and the morphism $\sum_{i=1}^m
x_i g_i \in \Hom (a,b)_{A{\mathcal G}}$ to the bounded linear map
$\sum_{i=1}^m \pi (x_i) \rho (g_i) \colon \rho (a) \rightarrow \rho
(b)$.

A proof of the following result can be found in \cite{Mitch6}.

\begin{proposition} \label{allcor}
Let $A$ be a $\mathcal G$-$C^\ast$-algebra.  Then any
$C^\ast$-functor $A{\mathcal G}\rightarrow {\mathcal L}$ takes the
form $(\rho , \pi )_\ast$ for some covariant representation $(\rho ,
\pi )$. \noproof
\end{proposition}

Let $A$ be a $\mathcal G$-$C^\ast$-algebra.  Fix an object $a\in \Ob
({\mathcal G})$, and choose a representation $\alpha \colon
A(a)\rightarrow {\mathcal L}(H)$ on some Hilbert space $H$.  For
each object $b\in \Ob ({\mathcal G})$, let $l^2 (a,b)$ be the
Hilbert space consisting of all sequences $(\eta_g )_{g\in \Hom
(a,b)_{\mathcal G}}$ in the space $H$ such that the series
$\sum_{g\in \Hom (a,b)_{\mathcal G}} \| \eta_g \|^2$ converges.

We can define a unitary representation of the groupoid $\mathcal G$
by mapping the object $b\in \Ob ({\mathcal G})$ to the Hilbert space
$l^2 (a,b)$, and the morphism $g\in \Hom (b,c)_{\mathcal G}$ to the
opertator $\rho (g) \colon l^2 (a,b)\rightarrow l^2 (a,c)$ defined
by translation.

There are corresponding representations $\pi \colon A(b) \rightarrow {\mathcal L}(l^2 (a,b))$ defined by writing
\[
\pi (x)((\eta_g )_{g\in \Hom (a,b)_{\mathcal G}}) = (\alpha (g^{-1}(x))\eta_g )_{g\in \Hom (a,b)_{\mathcal G}}
\]

It is straightforward to verify that the pair $(\rho, \pi )$ is a covariant representation of $A$.

\begin{definition}
A covariant representation of the type described above is called a {\em regular representation}.
\end{definition}

It is shown in \cite{Mitch6} that we can define $C^\ast$-norms $\| - \|_\mathrm{max}$ and $\| - \|_r$ on the algebroid $A{\mathcal G}$ by writing
$$\| \mu \|_\mathrm{max} = \sup \{ (\rho , \pi )_\ast (\mu ) \ |\ (\rho , \pi ) \textrm{ is a covariant representation of $A$ } \}$$
and
$$\| \mu \|_r = \sup \{ (\rho , \pi )_\ast (\mu ) \ |\ (\rho , \pi ) \textrm{ is a regular representation of $A$ } \}$$
respectively for any morphism $\mu$ in the category $A{\mathcal G}$.

\begin{definition}
The {\em full crossed product}, $A\rtimes {\mathcal G}$ is the
$C^\ast$-category defined by completion of the algebroid $A{\mathcal
G}$ with respect to the norm $\| - \|_\mathrm{max}$.

The {\em reduced crossed product}, $A\rtimes_r {\mathcal G}$ is the
$C^\ast$-category defined by completion of the algebroid $A{\mathcal
G}$ with respect to the norm $\| - \|_r$.
\end{definition}

As a special case of the above construction, for any groupoid
$\mathcal G$, we can define (see \cite{DL} and \cite{Mitch2}) the
full and reduced $C^\ast$-categories $C^\ast {\mathcal G} = {\mathbb
C}\rtimes {\mathcal G}$ and $C^\ast_r {\mathcal G} = {\mathbb
C}\rtimes_r {\mathcal G}$ respectively.

Let $\mathcal G$ be a groupoid, and let $\alpha \colon A\rightarrow
B$ be an equivariant map between $\mathcal G$-$C^\ast$-algebras.
Then we have an induced $C^\ast$-functor $\alpha_\ast \colon
A{\mathcal G}\rightarrow B{\mathcal G}$ defined to be the identity
on the set of objects, and by the formula
\[
f_\ast \left( \sum_{i=1}^m x_i g_i \right) = \sum_{i=1}^m \alpha (x_i) g_i \qquad x_i \in A(b),\ g_i \in \Hom (a,b)_{\mathcal G}
\]
on morphism sets.  This functor is continuous with respect to either norm.

Let $f\colon {\mathcal G}\rightarrow {\mathcal H}$ be a functor
between groupoids, and let $A$ be a $\mathcal H$-$C^\ast$-algebra.
Then we have an induced $C^\ast$-functor $f_\ast \colon A{\mathcal
G}\rightarrow A{\mathcal H}$ defined to be the functor $f$ on the
set of objects, and by the formula
\[
f_\ast \left( \sum_{i=1}^m x_i g_i \right) = \sum_{i=1}^m x_i f(g_i) \qquad x_i \in A(b),\ g_i \in \Hom (a,b)_{\mathcal G}
\]
on morphism sets.

This functor is continuous with respect to the norm $\|
-\|_\mathrm{max}$, and continuous with respect to the norm $\| -
\|_r$ if the functor is faithful.  We thus have the following
result.

\begin{proposition}
Let $\mathcal G$ be a groupoid.  Then the assignments $A\mapsto
A\rtimes_r {\mathcal G}$ and $A\mapsto A\rtimes {\mathcal G}$ are
functors from the category of $\mathcal G$-$C^\ast$-algebras and
equivariant maps to the category of $C^\ast$-categories and
$C^\ast$-functors.

Let $f\colon {\mathcal G}\rightarrow {\mathcal H}$ be a functor
between groupoids, and let $A$ be an $\mathcal H$-$C^\ast$-algebra.
Then we have a functorially induced $C^\ast$-functor $f_\ast \colon
A\rtimes {\mathcal G}\rightarrow A\rtimes {\mathcal H}$.  If the
functor $f$ is faithful, then we also have a functorially induced
$C^\ast$-functor $f_\ast \colon A\rtimes_r {\mathcal G}\rightarrow
A\rtimes_r {\mathcal H}$. \noproof
\end{proposition}

\begin{theorem} \label{descente}
Let $\mathcal G$ be a groupoid, and let $A$ and $B$ be $\mathcal G$-$C^\ast$-algebras.  Then there are maps
\[
D\colon \KK_{\mathcal G} (A,B)\rightarrow \KK(A\rtimes {\mathcal G},
B\rtimes {\mathcal G}) \qquad D_r\colon \KK_{\mathcal G}
(A,B)\rightarrow \KK(A\rtimes_r {\mathcal G}, B\rtimes_r {\mathcal
G})
\]
which compatible with the product in the sense that
we have commutative diagrams
\[
\begin{array}{ccc}
\KK_{\mathcal G}({A},{B})\wedge \KK_{\mathcal G}({B},{C}) & \rightarrow & \KK (A,{C}) \\
\downarrow & & \downarrow \\
\KK ({A}\rtimes {\mathcal G},{B}\rtimes {\mathcal G})\wedge \KK ({B}\rtimes {\mathcal G},{C}\rtimes {\mathcal G}) & \rightarrow & \KK ({A}\rtimes {\mathcal G},{C}\rtimes {\mathcal G}) \\
\end{array}
\]
and
\[
\begin{array}{ccc}
\KK_{\mathcal G}({A},{B})\wedge \KK_{\mathcal G}({B},{C}) & \rightarrow & \KK (A,{C}) \\
\downarrow & & \downarrow \\
\KK ({A}\rtimes_r {\mathcal G},{B}\rtimes_r {\mathcal G})\wedge \KK ({B}\rtimes_r {\mathcal G},{C}\rtimes_r {\mathcal G}) & \rightarrow & \KK ({A}\rtimes_r {\mathcal G},{C}\rtimes_r {\mathcal G}) \\
\end{array}
\]
where the horizontal maps are defined by the product, and the
vertical maps are copies of the map $D$ or $D_r$ respectively.
\end{theorem}

\begin{proof}
As in theorem \ref{sstens}, we can show that we have a semi-split
exact sequence
$$0 \rightarrow (JA)\rtimes_r {\mathcal
G}\stackrel{i_\ast}{\rightarrow} (T{A})\rtimes_r {\mathcal G}
\stackrel{\pi_\ast}{\rightarrow} {A}\rtimes_r {\mathcal G}
\rightarrow 0$$

We therefore have a natural $C^\ast$-functor $\gamma \colon
J(A\rtimes_r {\mathcal G})\rightarrow (JA)\rtimes_r {\mathcal G}$
defined as the classifying map of the diagram
\[
\begin{array}{ccccccccc}
& & & & & & {A}\rtimes_r {\mathcal G} \\
& & & & & & \| \\
0 & \rightarrow & (J{A})\rtimes_r {\mathcal G} & \rightarrow & (T{A})\rtimes_r {\mathcal G} & \rightarrow & {A}\rtimes_r {\mathcal G} & \rightarrow & 0 \\
\end{array}
\]

Viewing the suspension of a $\mathcal G$-$C^\ast$-algebra or
$C^\ast$-category as the tensor product with the $C^\ast$-algebra
$C_0 (0,1)$, there is an obvious $C^\ast$-functor $\beta \colon
(\Sigma^n B\otimes {\mathcal K}){\mathcal G}\rightarrow \Sigma^n
(B\rtimes_r {\mathcal G})\otimes {\mathcal K}$.

Let $\mathcal C$ be any $C^\ast$-category.  Then the tensor product
${\mathcal C}\otimes {\mathcal K}$ is a direct limit of
$C^\ast$-categories of matrices with elements the morphisms of the
category $\mathcal C$.  It follows that we have a natural
homomorphism $\delta \colon {\mathcal C}\otimes {\mathcal
K}\rightarrow {\mathcal C}^H$.

Let $\alpha \colon J^{2n}A\rightarrow \Sigma^n B\otimes {\mathcal
K}$ be an equivariant map.  Then by the above proposition we have a
functorially induced homomorphism $\alpha_\ast \colon
(J^{2n}A)\rtimes_r {\mathcal G} \rightarrow (\Sigma^n B\otimes
{\mathcal K})\rtimes_r {\mathcal G}$.  We can define a map $D\colon
\KK_{\mathcal G} (A,B)\rightarrow \KK (A\rtimes_r {\mathcal G},
B\rtimes_r {\mathcal G})$ by writing $D (\alpha ) = \delta \circ
\beta \circ \alpha_\ast \circ \gamma^n$.  The relevant naturality
properties are easy to check.

The argument for the result when considering the maximal crossed
product, $\rtimes_\mathrm{max}$, is identical to the above.
\end{proof}

\begin{corollary}
Let $\mathcal G$ be a discrete groupoid, and let $A$ and $B$ be
$\mathcal G$-$C^\ast$-algebras.  Then the spectra $\KK (A\rtimes_r
{\mathcal G}, B\rtimes_r {\mathcal G})$ and $\KK (A\rtimes {\mathcal
G},B\rtimes {\mathcal G})$ are symmetric $\KK_{\mathcal G}(\F ,\F
)$-module spectra. \noproof
\end{corollary}

\section{Comparison with $C^\ast$-algebra $K$-theory}

The spectra defined in this article are based on the spaces used to
construct $KK$-theory groups in the articles \cite{Cu3,Cu4,Cu5}. The
fact that our spectra can be used to define the usual Kasparov
$KK$-theory for $C^\ast$-algebras is therefore no surprise.  To be
specific, the following result holds.

\begin{theorem}
Let $A$ and $B$ be $C^\ast$-algebras.  Then the stable homotopy
group $\pi_n \KK (A,B)$ is naturally isomorphic to the group
$KK^{-n}(A,B)$.  If $C$ is another $C^\ast$-algebra, the smash
product of spectra
\[
\KK (A,B)\wedge \KK (B,C) \rightarrow \KK (A,C)
\]
induces the Kasparov product.
\end{theorem}

\begin{proof}
By proposition \ref{AHCA}, the $k$-th space of the spectrum $\KK
(A,B)$ is the spaces of all $\ast$-homomorphisms $J^{2k}
A\rightarrow \Sigma^k B\otimes {\mathcal K}$.  Let $[C,D]$ denote
the set of homomotopy classes of $\ast$-homomorphisms between
$C^\ast$-algebras $C$ and $D$.  Then for all $p\in {\mathbb N}$,
$\pi_p [C,D] = [C,\Sigma^pD]$, and the stable homotopy group $\pi_n
\KK (A,B)$ is thus the direct limit
\[
\lim_{k\rightarrow \infty}[J^{2n+2k} A,\Sigma^{n+k}B\otimes {\mathcal K}]
\]

The result now follows by proposition 3.1 in \cite{Cu5}.
\end{proof}

A similar result also holds in the equivariant case, and also
follows from the arguments of \cite{Cu5}.   To be precise, we have
the following.

\begin{theorem} \label{KKid}
Let $G$ be a discrete group, and let $A$ and $B$ be
$G$-$C^\ast$-algebras.  Then the stable homotopy group $\pi_n \KK_G
(A,B)$ is naturally isomorphic to the group $KK_G^{-n}(A,B)$.  If
$C$ is another $G$-$C^\ast$-algebra, the smash product of spectra
\[
\KK_G (A,B)\wedge \KK_G (B,C) \rightarrow \KK_G (A,C)
\]
induces the Kasparov product.
\noproof
\end{theorem}

In order to apply the above two theorems to spectra for
$C^\ast$-categories and groupoid $C^\ast$-algebras, we need further
comparison results.

\begin{definition}
Let $\mathcal A$ and $\mathcal B$ be unital $C^\ast$-categories. Two
$\ast$-functors $\alpha ,\beta \colon {\mathcal A}\rightarrow
{\mathcal B}$ are termed {\em equivalent} if there are elements $u_a
\in \Hom (\alpha (a),\beta (a))_{\mathcal B}$ for each object $a\in
\Ob ({\mathcal A})$ such that:

\begin{itemize}

\item $u_a^\ast u_a = 1_{\alpha (a)}$ and $u_a u_a^\ast = 1_{\beta (a)}$ for all objects $a\in \Ob ({\mathcal A})$.

\item Let $x\in \Hom (a,b)_{\mathcal A}$.  Then $\beta (x)u_a = u_b \alpha (x)$.
\end{itemize}

Two unital $C^\ast$-categories $\mathcal A$ and $\mathcal B$ are
called equivalent if there is are $\ast$-functors $\alpha \colon
{\mathcal A}\rightarrow {\mathcal B}$ and $\beta \colon {\mathcal
A}\rightarrow {\mathcal B}$ such that the compositions $\alpha \circ
\beta$ and $\beta \circ \alpha$ are equivalent to identity
$\ast$-functors.
\end{definition}

\begin{lemma} \label{h-lemma}
Let $\alpha , \beta \colon {\mathcal A}\rightarrow {\mathcal B}$ be
equivalent $\ast$-functors.  Then $\alpha$ and $\beta$ lie in the
same path-component of the space $F({\mathcal A},{\mathcal B})$.
\end{lemma}

\begin{proof}
In the space $F({\mathcal A},{\mathcal B})$, the $\ast$-functors
$\alpha$ and $\beta$ are the same as the $\ast$-functors $\alpha'
\colon {\mathcal A}\rightarrow {\mathcal B}_\oplus$ and $\beta'
\colon {\mathcal A}\rightarrow {\mathcal B}_\oplus$ defined by
writing
\[
\alpha'(a) = \alpha (a) \oplus \beta (a),\ \beta'(a) = \alpha (a) \oplus \beta (a) \qquad a\in \Ob ({\mathcal A})
\]
and
\[
\alpha' (x) = \left( \begin{array}{cc}
\alpha (x) & 0 \\
0 & 0 \\
\end{array} \right) \qquad
\beta' (x) = \left( \begin{array}{cc}
0 & 0 \\
0 & \beta (x) \\
\end{array} \right)
\]
where $x\in \Hom (a,b)_{\mathcal A}$.

Since the $\ast$-functors $\alpha$ and $\beta$ are equivalent, we
can find morphisms $u_a \in \Hom (\alpha (a),\beta (a))_{\mathcal
B}$ for each object $a\in \Ob ({\mathcal A})$ such that $u_a^\ast
u_a = 1_{\alpha (a)}$, $u_a u_a^\ast = 1_{\beta (a)}$, and $\beta
(x)u_a = u_b \alpha (x)$ for all $x\in \Hom (a,b)_{\mathcal A}$.

Let $t\in [0, \pi /2]$.  Define
\[
r_a (t) = \left( \begin{array}{cc}
\cos \theta & u_a^\ast \sin \theta \\
-u_a\sin \theta & \cos \theta \\
\end{array} \right) \in \Hom (\alpha (a)\oplus \beta (a))_{\mathcal B} \qquad t\in [0,\frac{\pi}{2}]
\]

Then we have a path of $\ast$-functors, $F_t\colon {\mathcal A}\rightarrow {\mathcal B}_\oplus$, from $\alpha'$ to $\beta'$, defined by the formula
\[
F_t (a) = \alpha (a)\oplus \beta (a) \qquad F_t (x) = r_b(t) \left( \begin{array}{cc}
\alpha (x) & 0 \\
0 & 0 \\
\end{array} \right) r_a(t)^\ast
\]
where $x\in \Hom (a,b)_{\mathcal A}$.
\end{proof}

\begin{theorem}
Let $\mathcal A$ and ${\mathcal A}'$ be equivalent
$C^\ast$-categories.  Let $\mathcal B$ be another $C^\ast$-category.
Then the spectra $\KK ({\mathcal A},{\mathcal B})$ and $\KK
({\mathcal A}',{\mathcal B})$ are homotopy-equivalent, and the
spectra $\KK ({\mathcal B},{\mathcal A})$ and $\KK ({\mathcal
B},{\mathcal A}')$ are homotopy-equivalent.
\end{theorem}

\begin{proof}
Since the $C^\ast$-categories $\mathcal A$ and ${\mathcal A}'$ are
equivalent, by the above lemma we can find elements $\alpha \in
F({\mathcal A},{\mathcal A}')$ and $\beta \in F({\mathcal
A}',{\mathcal A})$ along with a path $\gamma_t \in F({\mathcal
A},{\mathcal A})$ such that $\gamma_0 = \beta \circ \alpha$ and
$\gamma_1$ is the identity.

There are thus induced maps
\[
\alpha \sharp \colon \KK ({\mathcal A},{\mathcal B})\rightarrow \KK
({\mathcal A}',{\mathcal B})\qquad \beta \sharp \colon \KK
({\mathcal A}',{\mathcal B})\rightarrow \KK ({\mathcal A},{\mathcal
B})
\]
defined by the product with the elements $\alpha$ and $\beta$
respectively such that the map $\gamma_t \sharp \colon \KK
({\mathcal A},{\mathcal B})\rightarrow \KK ({\mathcal A},{\mathcal
B})$ is a homotopy between the composite $\alpha'\circ \alpha$ and
the identity map.

Similarly, the composite $\beta \circ \alpha$ is homotopic to the
identity map.  It follows that the spectra $\KK ({\mathcal
A},{\mathcal B})$ and $\KK ({\mathcal A}',{\mathcal B})$ are
homotopy-equivalent, and we have proved the first of the statements
in the theorem.

The proof of the second statement in the theorem is almost identical.
\end{proof}

The above result along with theorem \ref{KKid} can be used to prove
certain formal properties involving the $KK$-theory of
$C^\ast$-categories that are equivalent to $C^\ast$-algebras, which
covers most examples found in geometric applications.

\begin{theorem}
Let $\theta \colon {\mathcal G}\rightarrow {\mathcal H}$ be an
equivalence of discrete groupoids.  Let $A$ and $B$ be unital
$\mathcal H$-$C^\ast$-algebras.  Then the restriction map
$\theta^\ast \colon \KK_{\mathcal H}(A,B)\rightarrow \KK_{\mathcal
G}(A,B)$ is a homeomorphism of spectra.
\end{theorem}

\begin{proof}
Since the functor $\theta$ is an equivalence, there is a functor
$\phi \colon {\mathcal H}\rightarrow {\mathcal G}$ along with
natural isomorphisms $G\colon \phi \circ \theta \colon 1_{\mathcal
G}$ and $H\colon \theta \circ \phi \colon 1_{\mathcal H}$.

Thus, for each object $a\in \Ob ({\mathcal G})$, there is a morphism
$G_a \in \Hom (\phi \theta (a),a)_{\mathcal G}$.  Let $\alpha \colon
A\rightarrow B\otimes {\mathcal K}$ be an $\mathcal H$-equivariant
map.  Then the map $\alpha$ can be defined in terms of the
restriction $\phi^\ast \theta^\ast \alpha \colon A\rightarrow
B\otimes {\mathcal K}$ by the formula
\[
\alpha (x) = \phi^\ast \theta^\ast \alpha (H(a)^{-1} xH(a)) \qquad a\in \Ob ({\mathcal A})
\]

Thus the equivariant map $\alpha$ is determined by the restriction
$\phi^\ast \theta^\ast \alpha$.  The natural homomorphism $H$
therefore induces a homeomorphism of spectra $H_\ast \colon
\KK_{\mathcal H}(A,B)\rightarrow \KK_{\mathcal H}(A,B)$ such that
$H_\ast \circ \phi^\ast \circ \theta^\ast = 1_{\KK_{\mathcal
H}(A,B)}$.  There is similarly a homeomorphism $G_\ast \colon
\KK_{\mathcal G}(A,B)\rightarrow \KK_{\mathcal G}(A,B)$ such that
$G_\ast \circ \theta^\ast \circ \phi^\ast = 1_{\KK_{\mathcal
G}(A,B)}$.

It follows that the map $\theta^\ast$ is a homeomorphism, and we are done.
\end{proof}

\bibliographystyle{plain}

\end{document}